\documentclass[a4paper, 11pt]{amsart}
\usepackage{eucal}
\usepackage{enumerate}
\usepackage{tikz-cd}
\usepackage{xcolor}
\usepackage[colorlinks=true,linkcolor=blue,pagebackref,breaklinks]{hyperref}

\theoremstyle{plain}
\newtheorem{thm}{Theorem}[subsection]
\newtheorem*{thm*}{Theorem}
\newtheorem*{thm1}{Theorem 1}
\newtheorem*{thm2}{Theorem 2}
\newtheorem{prop}{Proposition}[subsection]
\newtheorem{lemma}{Lemma}[subsection]
\newtheorem*{lemma*}{Lemma}
\newtheorem{coro}{Corollary}[subsection]
\newtheorem*{coro*}{Corollary}
\newtheorem{conj}{Conjecture}[subsection]

\theoremstyle{definition}
\newtheorem{dfn}{Definition}[subsection]
\newtheorem*{dfn*}{Definition}
\newtheorem{rem}{Remark}[subsection]
\newtheorem*{rem*}{Remark} 

\newtheorem*{ex*}{Example}
\newtheorem{hyp}{Hypothesis}[subsection]

\newcommand{\A}{\mathbb A}
\newcommand{\C}{\mathbb{C}}
\newcommand{\R}{\mathbb R}
\newcommand{\Z}{\mathbb Z}
\newcommand{\Q}{\mathbb Q}

\newcommand{\Gal}{\operatorname{Gal}}
\newcommand{\Hom}{\operatorname{Hom}}
\newcommand{\Aut}{\operatorname{Aut}}

\newcommand{\Res}{\operatorname{Res}}
\newcommand{\GL}{\operatorname{GL}}

\newcommand{\Gm}[1]{\mathbb{G}_{\operatorname{m},#1}}

\newcommand{\St}{\operatorname{St}}
\newcommand{\Sh}{\operatorname{Sh}}
\newcommand{\hol}{\operatorname{hol}}

\newcommand{\can}{\operatorname{can}}

\newcommand{\mot}{\operatorname{mot}}

\newcommand{\cm}{F}
\newcommand{\tr}{F^{+}}

\newcommand*\diff{\mathop{}\!\mathrm{d}}
\newcommand{\AQ}{\mathbb{A}}
\newcommand{\Acm}{\mathbb{A}_{F}}

\newcommand{\Atr}{\mathbb{A}_{F^{+}}}

\newcommand{\Icm}{\mathcal{I}_{F}}
\newcommand{\Pet}{\operatorname{Pet}}

\numberwithin{equation}{subsection}

\begin{document} 

\title[]{Galois equivariance of critical values of $L$-functions for unitary groups}
\author{Lucio Guerberoff}
\address{Department of Mathematics, University College London, 25 Gordon Street, London WC1H 0AY, UK}
\email{l.guerberoff@ucl.ac.uk}
\author{Jie Lin}
\address{Institute des Hautes \'Etudes Scientifiques, 91440 Bures-sur-Yvette, France}
\email{linjie@ihes.fr}

\thanks{J.L. was supported by the European Community's Seventh Framework Programme (FP7/2007-2013) / ERC Grant agreement no. 290766 (AAMOT)}

\begin{abstract} The goal of this paper is to provide a refinement of a formula proved by the first author which expresses some critical values of automorphic $L$-functions on unitary groups as Petersson norms of automorphic forms. Here we provide a Galois equivariant version of the formula. We also give some applications to special values of automorphic representations of $\GL_{n}\times\GL_{1}$. We show that our results are compatible with Deligne's conjecture. 
\end{abstract}

\maketitle

\tableofcontents

\section{Introduction}\label{sec:intro}

In the present paper we provide a Galois equivariant version of a formula for the critical values of $L$-functions of cohomological automorphic representations of unitary groups. Such formula expresses the critical values in terms of Petersson norms of holomorphic automorphic forms, and was proved by Harris (\cite{harriscrelle}) when the base field is $\Q$, and by the first author when the base field is a general totally real field (\cite{guerbperiods}). To state the main theorem, we need to introduce some notation. Let $\cm/\tr$ be a CM extension, and let $G$ be a similitude unitary group attached to an $n$-dimensional hermitian vector space over $\cm$. Fix a CM type $\Phi$ for $\cm/\tr$, and let $(r_{\tau},s_{\tau})_{\tau \in \Phi}$ be the signature of $G$. Let $\pi$ be a cohomological, cuspidal automorphic representation of $G(\A)$. The weight of $\pi$ can be parametrized by a tuple of integers $\left((a_{\tau,1},\dots,a_{\tau,n})_{\tau\in\Phi};a_{0}\right)$. We let $\psi$ be an algebraic Hecke character of $\cm$, with infinity type $(m_{\tau})_{\tau:\cm\hookrightarrow\C}$. Under some additional hypotheses on $\pi$, it is proved in Theorem 4.5.1 of \cite{guerbperiods} that
\begin{eqnarray}
\label{formula up to FGal} &L^{S}\left(m-\frac{n-1}{2},\pi\otimes\psi,\St\right)&\\
\nonumber &\sim_{E(\pi,\psi);\cm^{\Gal}}(2\pi i)^{[\tr:\Q](mn-n(n-1)/2)-2a_{0}}D_{\tr}^{\lfloor\frac{n+1}{2}\rfloor/2}P(\psi)Q^{\hol}(\pi) &
\end{eqnarray}
for certain integers $m>n$ satisfying an inequality determined by the signatures of $G$ and the weight of $\pi$. In this expression, $\sim_{E(\pi,\psi);\cm^{\Gal}}$ means that the elements on each side, which belong to $E(\pi,\psi)\otimes\C$, differ by an element of $E(\pi,\psi)\otimes\cm^{\Gal}$. Here $E(\pi,\psi)=E(\pi)\otimes E(\psi)$, where $E(\pi)$ and $E(\psi)$ are certain number fields explicitly attached to $\pi$ and $\psi$, and $\cm^{\Gal}$ is the Galois closure of $\cm$ in $\C$. For simplicity we assume that $E(\pi)$ contains $\cm^{\Gal}$ in the following. The element $P(\psi)$ is an explicit expression involving CM periods attached to $\psi$, and $Q^{\hol}(\pi)$ is an automorphic quadratic period, which is basically given as the Petersson norm of an arithmetic holomorphic vector in $\pi$. It turns out that, up to multiplication by an element in $E(\pi,\psi)\otimes\cm^{\Gal}$, the product $(2\pi i)^{-2a_{0}}P(\psi)Q^{\hol}(\pi)$ can be seen as the inverse of a Petersson norm of an arithmetic vector in $\pi\otimes\psi$ contributing to antiholomorphic cohomology. In this paper, we will consider a Galois equivariant version of formula (\ref{formula up to FGal}) when using these inverse Petersson norms, which we denote by $Q(\pi,\psi)$ in this introduction, for fixed choices of arithmetic vectors; we refer the reader to Subsection \ref{subsection: petersson} for more details. Galois equivariance means that we obtain a formula up to factors in $E(\pi,\psi)$ instead of $E(\pi,\psi)\otimes\cm^{\Gal}$. We also incorporate an auxiliary algebraic Hecke character $\alpha$, which will provide useful applications. The infinity type of $\alpha$ will be assumed to be given by an integer $\kappa$ at all places of $\Phi$, and by $0$ at places outside $\Phi$.

The formula (\ref{formula up to FGal}) is proved using the doubling method, and it relies on a detailed analysis of certain global and local zeta integrals. In this paper, we study the action of $\Gal(\bar\Q/\Q)$ on these zeta integrals. The global and the finite zeta integrals are not hard to analyse, but the archimedean zeta integral is subtler. This integral depends on certain choices that will not be explicited in this introduction, but most importantly, it depends on $\pi$, $\psi$, $\alpha$ and the integer $m$. We denote it by $Z_{\infty}(m;\pi,\psi,\alpha)$ here.  Garrett proved in \cite{garrett} that $Z_{\infty}(m;\pi,\psi,\alpha)$ is non-zero and belongs to $F^{\Gal}$, so it doesn't appear in (\ref{formula up to FGal}), but at the moment we must include it in our Galois equivariant formulation.

Besides the archimedean integral, there is another factor that needs to be added to (\ref{formula up to FGal}) to obtain a Galois equivariant version, which is not originally visible since it belongs to $\cm^{\Gal}$. To the quadratic extension $\cm/\tr$, there is attached a quadratic character $\varepsilon_{\cm}$ of $\Gal(\bar\cm/\tr)$ and an Artin motive $[\varepsilon_{\cm}]$ over $\tr$. We let $\delta[\varepsilon_{\cm}]$ be the period of this motive, an element of $\C^{\times}$ well defined up to multiplication by an element in $\Q^{\times}$. It can also be seen as $c^{-}[\varepsilon_{L}]$. When $\tr=\Q$, it can be explicitly written down as a classical Gauss sum. In any case, $\delta[\varepsilon_{\cm}]\in\cm^{\Gal}$.


We then define
\[ \mathfrak{Q}(m;\pi,\psi,\alpha)=\frac{Q(\pi,\psi,\alpha)}{Z_{\infty}(m;\pi,\psi,\alpha)}. \]
We can define $L^{*}(s,\pi\otimes\psi,\St,\alpha)\in E(\pi,\psi,\alpha)\otimes\C$ to be the collection of standard $L$-functions of ${}^{\sigma}\pi\otimes{}^{\sigma}\psi$, twisted by ${}^{\sigma}\alpha$, for $\sigma:E(\pi,\psi,\alpha)\hookrightarrow\C$. The automorphic representation ${}^{\sigma}\pi$ and the Hecke characters ${}^{\sigma}\psi$ and ${}^{\sigma}\alpha$ are obtained from $\pi$, $\psi$ and $\alpha$ by conjugation by $\sigma$ (see Subsections \ref{subsection: conjugation} and \ref{subsection: algebraic Hecke} for details). We can similarly define $\mathfrak{Q}^{*}(m;\pi,\psi,\alpha)\in E(\pi,\psi,\alpha)\otimes\C$. The main result of this paper is the following.

\begin{thm1} Keep the assumptions as above. Let $m>n-\frac{\kappa}{2}$ be an integer satisfying inequality (\ref{mainineq}). Then
	\[ L^{*,S}\left(m-\frac{n}{2},\pi\otimes\psi,\St,\alpha\right)\sim_{E(\pi,\psi,\alpha)}\]\[(2\pi i)^{[\tr:\Q](mn-n(n-1)/2)}D_{\tr}^{-\lfloor\frac{n}{2}\rfloor/2}\delta[\varepsilon_{\cm}]^{\lfloor\frac{n}{2}\rfloor}\mathfrak{Q}^{*}(m;\pi,\psi,\alpha).\]
\end{thm1}
The presence of $m$ in the element  $\mathfrak{Q}^{*}(m;\pi,\psi,\alpha)$ is, as we explained above, due to the difficulty in analyzing the Galois action on the archimedean zeta integrals. If we all factors in $E(\pi,\psi,\alpha)\otimes\cm^{\Gal}$, then we can replace $\mathfrak{Q}^{*}(m;\pi,\psi,\alpha)$ with the period $(2\pi i)^{-2a_{0}}P(\psi;\alpha)Q^{\hol}(\pi)$, which becomes formula (\ref{formula up to FGal}) when $\alpha$ is trivial. In any case, we can at least stress that the dependence on $m$ of $\mathfrak{Q}(m;\pi,\psi,\alpha)\in E(\pi,\psi,\alpha)\otimes\C$ disappears if we see it modulo $E(\pi,\psi,\alpha)\otimes\cm^{\Gal}$.

As application, we consider the base change of $\pi$ and get a theorem on critical values of Rankin-Selberg $L$-function for $GL_{n}\times GL_{1}$ over a CM field. This is stated as a hypothesis in the thesis of the second author, and is used to prove an interesting factorisation of automorphic periods. It is also used there to deduce more results on critical values of Rankin-Selberg $L$-function.

Let $\Pi$ be a cuspidal automorphic representation of $GL_{n}(\Acm)$. We assume that for any $I\in \{0,1,\cdots.n\}^{\Phi}$ there exists $U_{I}$, a unitary group of rank $n$ with respect to $\cm/\tr$ of signature $(n-I_{\tau},I_{\tau})_{\tau \in \Phi}$, such that the representation $\Pi$ descends to an automorphic representation of $U_{I}$ which can be extended to an automorphic representation of the rational similitude unitary group $GU_{I}$ as in the above theorem. We define an automorphic period $P^{(I)}(\Pi)$ for each $I$. 
\begin{thm2} 
Let $\Pi$ be as above and $\eta$ be a Hecke character as in Theorem \ref{n*1}. There exists an explicit signature $I:=I(\Pi,\eta)$ such that if an integer $m\geq n-\cfrac{\kappa}{2}$ satisfies equation (\ref{mainineq}) with $s_{\tau}=I(\tau)$ and $r_{\tau}=n-I(\tau)$,
then we have:

\begin{eqnarray}
&L^{*}(m-\frac{n}{2},\Pi\otimes \eta) \sim_{E(\Pi,\eta);F^{\Gal}}& \\ \nonumber
 &(2 \pi i)^{(m-n/2)nd(\tr)} \Icm^{[n/2]} (D_{\tr}^{1/2})^{n}e_{\Phi}^{mn}P^{*,(I(\Pi,\eta))}(\Pi) \prod\limits_{\tau\in\Phi}p^{*}(\check{\eta},\tau)^{I(\tau)}p^{*}(\check{\eta},\overline{\tau})^{n-I(\tau)}&
\end{eqnarray}
where $d(\tr)$ is the degree of $\tr$ over $\Q$ and $p(\check{\eta},\tau)$ refers to the CM period.
\end{thm2}

At the end of the paper, we show that the above theorem is compatible with the Deligne conjecture for critical values of motives (c.f. \cite{deligne79}).


\subsection*{Acknowledgements} The authors would like to thank Michael Harris for his numerous suggestions and comments. We also want to thank Harald Grobner, Fabian Januszewski, Li Ma and Alberto Minguez for several useful conversations. 

\subsection*{Notation and conventions}
We fix an algebraic closure $\C$ of $\R$, a choice of $i=\sqrt{-1}$, and we let $\overline\Q$ denote the algebraic closure of $\Q$ in $\C$. We let $c\in\Gal(\C/\R)$ denote complex conjugation on $\C$, and we use the same letter to denote its restriction to $\overline\Q$. Sometimes we also write $c(z)=\overline z$ for $z\in\C$. We let $\Gamma_{\Q}=\Gal(\overline\Q/\Q)$.
For a number field $K$, we let $\A_{K}$ and $\A_{K,f}$ denote the rings of ad\`eles and finite ad\`eles of $K$ respectively. When $K=\Q$, we write $\A=\A_{\Q}$ and $\A_{f}=\A_{\Q,f}$.

A CM field $L$ is a totally imaginary quadratic extension of a totally real field $K$. A CM type $\Phi$ for $L/K$ is a choice of one of the two possible extensions to $L$ of each embedding of $K$.

All vector spaces will be finite-dimensional except otherwise stated. By a variety over a field $K$ we will mean a geometrically reduced scheme of finite type over $K$.
%
%
We let $\mathbb{S}=R_{\C/\R}\Gm{\C}$. We denote by $c$ the complex conjugation map on $\mathbb{S}$, so for any $\R$-algebra $A$, this is $c\otimes_{\R}1_A:(\C\otimes_{\R}A)^\times\to(\C\otimes_{\R}A)^\times$. We usually also denote it by $z\mapsto\overline z$, and on complex points it should not be confused with the other complex conjugation on $\mathbb{S}(\C)=(\C\otimes_{\R}\C)^{\times}$ on the second factor.

A tensor product without a subscript between $\Q$-vector spaces will always mean tensor product over $\Q$. For any number field $K$, we denote by $J_{K}=\Hom(K,\C)$. For $\sigma\in J_{K}$, we let $\overline\sigma=c\sigma$. 

Let $E$ be a number field and $L$ be a subfield of $\C$. If $z,w\in E\otimes \C$, we write $z\sim_{E; L}w$ if either $w=0$ or if $w\in(E\otimes \C)^{\times}$ and $z/w\in(E\otimes L)^{\times}$. When $L=\Q$, we simply write $a\sim_{E}b$. There is a natural isomorphism
$E\otimes\C\simeq\prod_{\varphi\in J_E}\C$ given by $e\otimes z\mapsto(\varphi(e)z)_{\varphi}$ for $e\in E$ and $z\in\C$. Under this identification, we denote an element $z\in E\otimes\C$ by $(z_{\varphi})_{\varphi\in J_E}$.


%
%
%
%
We choose Haar measures on local and adelic points of unitary groups as in the Introduction of \cite{harriscrelle}.

\section{Automorphic representations}\label{section: automorphic}
In this section we recall some basic facts about cohomological representation of a unitary group and their conjugation by $\Aut(\C)$.

\subsection{Unitary groups, Shimura varieties and conjugation}

Let $V$ be a hermitian space of dimension $n$ over $\cm$ with respect to $\cm/\tr$. We let $U$ be the (restriction of scalars from $\tr$ to $\Q$ of the) unitary group associated to $V$, and we let $G$ be the associated similitude unitary group with rational similitude factors. To be more precise, $U$ and $G$ are reductive algebraic groups over $\Q$, such that for any $\Q$-algebra $A$, the $A$-points are given as
\[U(A)=\{g\in\Aut_{\cm\otimes A}(V\otimes A) : gg^{*}=\operatorname{Id}\}
\]
and
\[G(A)=\{g\in\Aut_{\cm\otimes A}(V\otimes A) : gg^{*}=\mu(g)\operatorname{Id} \text{ with } \mu(g)\in A^{\times}\},
\]
where we write $g^{*}$ for the adjoint of $g$ with respect to the hermitian form.

We fix once and for all a CM type $\Phi$ for $\cm/\tr$. Attached to $G$ and $\Phi$ is a Shimura variety which we denote by $S=\Sh(G,X)$. The choice of $\Phi$ and an orthogonal basis of $V$ determine a choice of CM point $x\in X$, which will be fixed throughout the paper. We let $K_{x}\subset G_{\R}$ be the centralizer of $x$. For each compact open subgroup $K\subset G(\A_{f})$, $S_{K}$ will be the corresponding Shimura variety at level $K$. We also let $E(G,X)$ be the reflex field of $S$. For each $\tau\in J_{\cm}$, we let $(r_{\tau},s_{\tau})=(r_{\tau}(V),s_{\tau}(V))$ be the signature of $V$ at the place $\tau$. We can write the group $G_{\R}$ as
\begin{equation}\label{decomp GR} G_{\R}\cong G\left(\prod_{\tau\in\Phi}\operatorname{GU}(r_{\tau},s_{\tau})\right),\end{equation}
which is defined to be the set of tuples $(g_{\tau})_{\tau\in\Phi}$ that have the same similitude factor.

We will parametrize irreducible representations of $G_{\C}$ and of $K_{x,\C}$ by their highest weights, and we will use the conventions used in \cite{guerbperiods}, 3.3. Thus, an irreducible representation of $G_{\C}$ (resp. $K_{x,\C}$) will be given by a highest weight $\mu\in\Lambda^{+}$ (resp. $\lambda\in\Lambda_{c}^{+}$). The corresponding representations will be denoted by $W_{\mu}$ (resp. $V_{\lambda}$). All these parameters can be written as tuples
\[ \left((a_{\tau,1},\dots,a_{\tau,n})_{\tau\in\Phi};a_{0}\right) \]
where each $a_{\tau,i}$ and $a_{0}$ are integers, and $a_{\tau,1}\geq\dots\geq a_{\tau,n}$ for each $\tau\in\Phi$ in the case of irreducible representations of $G_{\C}$. In the case of representations of $K_{x,\C}$, the condition is that $a_{\tau,1}\geq\dots\geq a_{\tau,r_{\tau}}$ and $a_{\tau,r_{\tau}+1}\geq\dots\geq a_{\tau,n}$ for every $\tau\in\Phi$.

Let $\sigma\in\Aut(\C)$. We let $({}^{\sigma}G,{}^{\sigma}X)$ be the conjugate Shimura datum with respect to the automorphism $\sigma$ and the CM point $x\in X$ (see \cite{milne}, II.4, for details). It follow from \cite{milnesuh}, Theorem 1.3, that the group ${}^{\sigma}G$ can be realized as the unitary group attached to another $n$-dimensional hermitian space ${}^{\sigma}V$, whose signatures at infinity are obtained by permutation from those of $G$. More precisely,
\[ (r_{\tau}({}^{\sigma}V),s_{\tau}({}^{\sigma}V))=(r_{\sigma\tau}(V),s_{\sigma\tau}(V)) \]
for any $\tau\in J_{\cm}$. The local invariants of ${}^{\sigma}V$ at finite places are the same as those of $V$, and we identify ${}^{\sigma}G(\A_{f})$ with $G(\A_{f})$ without further mention.

 We can also conjugate automorphic vector bundles, as in \cite{milne}. The CM point $x\in X$ will be fixed throughout, and all conjugations will be with respect to this fixed point. For any $\sigma\in\Aut(\C)$, we have a CM point ${}^{\sigma}x\in{}^{\sigma}X$, and we let ${}^{\sigma}\Lambda^{+}$ and ${}^{\sigma}\Lambda_{c}^{+}$ denote the corresponding set of dominant weights for the groups ${}^{\sigma}G$ and $K_{{}^{\sigma}x}\subset{}^{\sigma}G_{\R}$. When $x$ needs to be specified, we will denote $\Lambda_{c}^{+}$ by $\Lambda_{c,x}^{+}$. Suppose that $\mathcal{E}_{\lambda}$ is a fully decomposed automorphic vector bundle over $S_{\C}$, associated with the irreducible representation of $K_{x,\C}$ with highest weight $\lambda$. Then $\mathcal{E}_{\lambda}\times_{\C,\sigma}\C$ is a vector bundle over $S_{\C}\times_{\C,\sigma}\C$, and identifying the latter with ${}^{\sigma}S_{\C}$, we get an automorphic vector bundle ${}^{\sigma}\mathcal{E}_{\lambda}$ over ${}^{\sigma}S_{\C}$. It is fully decomposed, associated with an irreducible representation of $K_{{}^{\sigma}x,\C}$ whose highest weight we denote by ${}^{\sigma}\lambda\in{}^{\sigma}\Lambda_{c}^{+}$.

\subsection{Conjugation of cohomological cuspidal representations}\label{subsection: conjugation}
From now on, we let $\pi=\pi_{\infty}\otimes\pi_{f}$ be an automorphic representation of $G(\A)$. We will assume that $\pi$ satisfies the following list of hypotheses.

\begin{hyp}\label{hypotheses}
\begin{enumerate}[(1)]
\item\label{discreteness} $\pi$ occurs in the discrete spectrum.
\item\label{cohomological} $\pi$ is cohomological with respect to some irreducible representation $W=W_{\mu}$ of $G_{\C}$, with $\mu\in\Lambda^{+}$.
\item \label{defoverQ} The representation $W$ is defined over $\Q$.
\item\label{temperedness} $\pi_{\infty}$ is essentially tempered.
\end{enumerate}
\end{hyp}

\begin{rem} By Theroem 4.3 of \cite{Wallach}, Hypotheses \ref{hypotheses} (\ref{discreteness}) and (\ref{temperedness}) imply that $\pi$ is cuspidal. This can also be deduced by assuming that $\pi_{v}$ is tempered at some finite place, as in Proposition 4.10 of \cite{clozel93}.
\end{rem}

\begin{rem} Hypothesis (\ref{defoverQ}) is assumed mostly for simplicity of notation.
\end{rem}

\begin{rem} If $\mu$ is regular, in the sense that $a_{\tau,i}>a_{\tau,i+1}$ for every $\tau\in\Phi$ and every $i=1,\dots,n-1$, then \ref{hypotheses} (\ref{discreteness}) and (\ref{cohomological}) imply (\ref{temperedness}) (see Prop. 4.2 and 5.2 of \cite{lischwermer} and Prop. 2.2 of \cite{schwermer}).
\end{rem}

Under these hypotheses, $\pi_{\infty}$ is a discrete series representation that belongs to the $L$-packet whose infinitesimal character is that of $W^{\vee}$.

By Theorem 4.4.1 of \cite{bhr}, the field of definition $\Q(\pi_{f})$ of the isomorphism class of $\pi_{f}$ is a subfield of a CM field. There is a finite extension $E_{0}(\pi)$ of $\Q(\pi_{f})$, which can also be taken to be a CM field, such that $\pi_{f}$ has a model $\pi_{f,0}$ over $E_{0}(\pi)$ (see Corollary 2.13 of \cite{harrisbeilinson}). We let ${}^{\sigma}\pi_{f}=\pi_{f}\otimes_{\C,\sigma}\C\cong\pi_{f,0}\otimes_{E_{0}(\pi),\sigma}\C$. We can also define the conjugate ${}^{\sigma}\pi_{\infty}$, a discrete series representation of ${}^{\sigma}G(\R)$, as in (2.19) of \cite{harrisbeilinson} (see also \cite{bhr}, 4.2). We remark that the restriction of ${}^{\sigma}\pi_{\infty}$ at $U({}^{\sigma}V)_{\tau}$ is isomorphic to the restriction of $\pi_{\infty}$ at $U(V)_{\sigma\tau}$ for any $\tau\in J_{F}$.

We will make the following assumption throughout the paper:
\begin{hyp}
There exists an automorphic representation ${}^{\sigma}\pi$ of ${}^{\sigma}G(\A)$ satisfying Hypotheses \ref{hypotheses} such that $({}^{\sigma}\pi)_{f}\cong{}^{\sigma}\pi_{f}$ (recall that we are identifying ${}^{\sigma}G(\A_{f})\cong G(\A_{f})$).
\end{hyp}


\begin{rem}
One of the main results of \cite{bhr} (Theorem 4.2.3) guarantees the existence of such ${}^{\sigma}\pi$ when the Harish-Chandra parameter of $\pi_{\infty}$ is far enough from the walls. Moreover, under these conditions, we have that $({}^{\sigma}\pi)_{\infty}\cong{}^{\sigma}\pi_{\infty}$. We will assume this throughout the paper. In particular, $({}^{\sigma}\pi)_{\infty}$ is cohomological of weight
\[ {}^{\sigma}\tilde\mu=\left(a_{\sigma\tau,1},\dots,a_{\sigma\tau,n}\right)_{\tau\in J_{\cm}} \]
In \cite{harrisbeilinson}, 4.3, further conditions under which ${}^{\sigma}\pi$ is shown to exist are discussed. A particular case of this is when the infinitesimal character of $\pi_{\infty}$ is regular and $\pi$ is not a CAP representation. This last condition is expected to be true for tempered representations. See also Corollary 2.14 of \cite{harrisbeilinson}.

\end{rem}

\subsection{The standard $L$-function and the motivic normalization}
Let $\pi$ be as above. As in \cite{harriscrelle}, 2.7, we can define the standard $L$-function of $\pi$ as $L^{S}(s,\pi,\St)=L^{S}(s,BC(\pi_{0}),\St)$. Here $\St$ refers to the standard representation of the $L$-group of $\GL_{n}$ over $L$, $\pi_{0}$ is an irreducible constituent of the restriction of $\pi$ to $U(\A)$, and $BC(\pi_{0})$ is the base change of $\pi_{0}$ to an irreducible admissible representation of $\GL_{n}(\A_{L}^{S})$, for a big enough finite set of places $S$ of $L$. The base change is defined locally at archimedean places, at split places, and at places of $K$ where the local unitary group $U_{v}$ and $\pi_{0,v}$ are unramified. Under our assumptions, it is known that $BC(\pi_{0})$ is the restriction to $\GL_{n}(\A_{L}^{S})$ of an automorphic representation $\Pi$ of $\GL_{n}(\A_{L})$, so we can actually define $L(s,\pi,\St)$ at all places as $L(s,\Pi,\St)$. We define the motivic normalization by
\[ L^{\mot,S}(s,\pi,\St)=L^{S}\left(s-\frac{n-1}{2},\pi,\St\right).\]
More generally, if $\alpha$ is an algebraic Hecke character of $\cm$, we define
\[ L^{S}(s,\pi,\St,\alpha)=L(s,BC(\pi_{0}),\St,\alpha),\]
the twisted $L$-function. The motivic normalization is defined similarly. We define
\[ L^{*,\mot,S}(s,\pi,\St,\alpha)=\left(L^{\mot,S}(s,{}^{\sigma}\pi,\St,{}^{\sigma}\alpha)\right)_{\sigma\in\Aut(\C)}.\]

\subsection{Algebraic Hecke characters}\label{subsection: algebraic Hecke} Let $\psi$ be an algebraic Hecke character of $\cm$, of infinity type $(m_{\tau})_{\tau\in J_{\cm}}$. Recall that this means that
\[ \psi:\A_{\cm}^{\times}/\cm^{\times}\to\C^{\times} \]
is continuous, and for each embedding $\tau\in J_{\cm}$, we have
\[ \psi(x)=\tau(x)^{-m_{\tau}}\bar\tau(x)^{-m_{\bar\tau}}\quad(x\in \cm_{w}^{\times}).\]
Here $w$ is the infinite place of $\cm$ determined by $\tau$. We let $\Q(\psi)$ be the field generated over $\Q$ by the values of $\psi$ on $\A_{\cm,f}^{\times}$. Then $\Q(\psi)$ is either $\Q$ or a CM field. If $\sigma\in\Aut(\C)$, we define ${}^{\sigma}\psi$ to be the algebraic Hecke character whose values on $\A_{\cm,f}^{\times}$ are obtained from those of $\psi$ by applying $\sigma$, and whose infinity type is $(m_{\sigma^{-1}\tau})_{\tau\in J_{\cm}}$.

We need to fix the following notation. Suppose that $\alpha_{0}$ is an algebraic Hecke character of $\tr$ of finite order, and $\sigma\in J_{\tr}$. Then
\[ \delta_{\sigma}[\alpha_{0}]\in(\Q(\alpha_{0})\otimes\C)^{\times} \]
is the $\delta$-period of the Artin motive $[\alpha_{0}]$. This is a motive over $F^{+}$ with coefficients in $\Q(\alpha_{0})$. We also let
\[ \delta[\alpha_{0}]=\delta_{1}\left(\Res_{\tr/\Q}[\alpha_{0}]\right)\in(\Q(\alpha_{0})\otimes\C)^{\times} \]
be the period of the motive $\Res_{\tr/\Q}[\alpha_{0}]$ obtained from $[\alpha_{0}]$ by restriction of scalars from $\tr$ to $\Q$. It is proved in \cite{yoshida} that
\[ \delta[\alpha_{0}]\sim_{\Q(\alpha_{0})}D_{\tr}^{\frac{1}{2}}\prod_{\sigma\in J_{\tr}}\delta_{\sigma}[\alpha_{0}].\]
Suppose now that $\alpha$ is an algebraic Hecke character of $\cm$ of weight $w$. Then we can write
\[ \alpha|_{\A_{\tr}^{\times}}=\alpha_{0}\|\cdot\|_{\A_{\tr}}^{-w},\]
where $\alpha_{0}$ is a finite order algebraic Hecke character of $\tr$. We define
\[ G(\alpha)=\prod_{\sigma\in J_{\tr}}\delta_{\sigma}[\alpha_{0}]\in(\Q(\alpha_{0})\otimes\C)^{\times}.\]
We can then write
\begin{equation}\label{delta0 G}
\delta[\alpha_{0}]\sim_{\Q(\alpha_{0})}D_{\tr}^{\frac{1}{2}}G(\alpha).
\end{equation}
For each embedding $\rho\in J_{\Q(\alpha_{0})}$, we let $G(\alpha)_{\rho}\in\C^{\times}$ be its $\rho$-component.


\section{The doubling method, conjugation and the main theorem}\label{section: main}
\subsection{Basic assumptions} In this section, we briefly recall the doubling method used to obtain the mail formula of \cite{guerbperiods}, and explain how it behaves under Galois conjugation. We fix once and for all a cuspidal automorphic representation $\pi$ of $G(\A)$, satisfying all the previous hypotheses. In particular, $\pi$ is cohomological of type $\mu=\left((a_{\tau,1},\dots,a_{\tau,n})_{\tau\in\Phi};a_{0}\right)$, with $W=W_{\mu}$ defined over $\Q$. We also assume that $\pi^{\vee}\cong\pi\otimes\|\nu\|^{2a_{0}}$, that $\pi$ contributes to antiholomorphic cohomology, and that
\begin{equation}\label{hypomult} \dim_{\C}\Hom_{\C[G(\A_{f})]}\left({}^{\sigma}\pi_{f},H^{d}_{!}({}^{\sigma}S_{\C},{}^{\sigma}\mathcal{E}_{\mu})\right) \leq1 \end{equation}
for all $\sigma\in\Aut(\C)$. This is part of Arthur's multiplicity conjectures for unitary groups, a proof of which is expected to appear in the near future. We refer the reader to \cite{kmsw} and their forthcoming sequels for more details.

We fix a CM type $\Phi$ for $\cm/\tr$, and an algebraic Hecke character $\psi$ of $\cm$ with infinity type $(m_{\tau})_{\tau\in J_{\cm}}$.
We let $\Lambda=\Lambda(\mu;\psi)\in\Lambda_{c}^{+}$ be the parameter
\[ \Lambda=\left((b_{\tau,1},\dots,b_{\tau,n})_{\tau\in\Phi};b_{0}\right),\]
where
\[ b_{\tau,i}=\left\{\begin{array}{ll} a_{\tau,s_{\tau}+i}+m_{\bar\tau}-m_{\tau}-s_{\tau} & \hbox{ if } 1\leq i\leq r_{\tau},
\\ a_{\tau,i-r_{\tau}}+m_{\bar\tau}-m_{\tau}+r_{\tau} & \hbox{ if } r_{\tau}+1\leq i\leq n,\end{array}\right. \]
and $b_{0}=a_{0}-n\sum_{\tau\in\Phi}m_{\bar\tau}$ (this was denoted by $\Lambda(\mu;\eta^{-1})$ in \cite{guerbperiods}). We similarly define
${}^{\sigma}\Lambda=\Lambda({}^{\sigma}\mu;{}^{\sigma}\psi)$ for $\sigma\in\Aut(\C)$.

\subsection{The double hermitian space} Given our hermitian space $V$, we let $-V$ be the hermitian space whose underlying $\cm$-vector space is $V$, but whose hermitian form is multiplied by $-1$. Its associated Shimura conjugacy class will be denoted by $X^{-}$. We let $2V=V\oplus-V$, and $(G^{(2)},X^{(2)})$ be the Shimura datum attached to $2V$. The choice of our CM point $x\in X$ gives rise to fixed CM points $x^{-}=\bar x\in X^{-}$ and $x^{(2)}\in X^{(2)}$. The reflex field of $(G^{(2)},X^{(2)})$ is $\Q$, and hence we can identify $({}^{\sigma}G^{(2)},{}^{\sigma}X^{(2)})=(G^{(2)},X^{(2)})$ for any $\sigma\in\Aut(\C)$. We let $S^{(2)}$ be the associated Shimura variety.

We also let $G^{\sharp}\subset G\times G$ be the subgroup of pairs with the same similitude factor, and we let $x^{\sharp}:\mathbb{S}\to G^{\sharp}_{\R}$ be the map $(x,x^{-})$. The corresponding Shimura datum will be denoted by $(G^{\sharp},X^{\sharp})$, and the Shimura variety by $S^{\sharp}$. There is a natural embedding
\[ i:(G^{\sharp},X^{\sharp})\to(G^{(2)},X^{(2)}) \]
of Shimura data, which induces a closed embedding of Shimura varieties $i:S^{\sharp}\to S^{(2)}$.

For $\sigma\in\Aut(\C)$, we let $({}^{\sigma}G)^{\sharp}\subset{}^{\sigma}G\times{}^{\sigma}G$ be the group defined in a similar fashion but using ${}^{\sigma}V$ and ${}^{\sigma}G$ instead of $V$ and $G$. Using the definition of the twisting, given for example in \cite{milne}, it is easy to see that we can naturally identify $({}^{\sigma}G)^{\sharp}$ with ${}^{\sigma}(G^{\sharp})$ as a subgroup of ${}^{\sigma}G\times{}^{\sigma}G$. We let ${}^{\sigma}i:{}^{\sigma}G^{\sharp}\hookrightarrow G^{(2)}$ be the inclusion defined above for ${}^{\sigma}V$.

Keep in mind that our CM point $x\in X$ is fixed, and this in turn gives choices of CM points ${}^{\sigma}x\in{}^{\sigma}X$ for each $\sigma$. We fix these CM points, as well as their variant $x^{(2)}$, $x^{\sharp}$, ${}^{\sigma}x^{\sharp}$ for varying $\sigma\in\Aut(\C)$. We will parametrize fully decomposed automorphic vector bundles over the corresponding Shimura varieties by irreducible representations of the corresponding groups $K_{x^{(2)},\C}$, $K_{x^{\sharp},\C}$, $K_{{}^{\sigma}x,\C}$. We have the following identifications:
\[ K_{x^{\sharp},\C}\cong\left(\prod_{\tau\in\Phi}GL_{r_{\tau},\C}\times\GL_{s_{\tau},\C}\times\GL_{s_{\tau},\C}\times\GL_{r_{\tau},\C}\right)\times\Gm{\C},\]
\[ K_{{}^{\sigma}x^{\sharp},\C}\cong\left(\prod_{\tau\in\Phi}GL_{r_{\sigma\tau},\C}\times\GL_{s_{\sigma\tau},\C}\times\GL_{s_{\sigma\tau},\C}\times\GL_{r_{\sigma\tau},\C}\right)\times\Gm{\C}, \]
\[ K_{x^{(2)},\C}\cong\left(\prod_{\tau\in\Phi}\GL_{n,\C}\times\GL_{n,\C}\right)\times\Gm{\C}.\]

For
\[ \lambda=\left((\lambda_{\tau,1},\dots,\lambda_{\tau,n})_{\tau\in\Phi};\lambda_{0}\right)\in\Lambda_{c,x}^{+} \]
and
\[ \lambda^{-}=\left((\lambda^{-}_{\tau,1},\dots,\lambda^{-}_{\tau,n})_{\tau\in\Phi};\lambda^{-}_{0}\right)\in\Lambda_{c,x^{-}}^{+},\]
we let
\[ (\lambda,\lambda^{-})^{\sharp}=\left((\lambda_{\tau,1},\dots,\lambda_{\tau,n},\lambda^{-}_{\tau,1},\dots,\lambda^{-}_{\tau,n})_{\tau\in\Phi};\lambda_{0}+\lambda^{-}_{0}\right)\in\Lambda_{c,x^{\sharp}}^{+}.\]
Let
\[ \lambda^{*}=\left((-\lambda_{\tau,n},\dots,-\lambda_{\tau,1})_{\tau\in\Phi};-\lambda_{0}\right)\in\Lambda_{c,x^{-}}^{+}. \]
For an integers $\kappa$, we define $\lambda^{\sharp}[\kappa]\in\Lambda_{c,x^{\sharp}}^{+}$ as $\lambda^{\sharp}[\kappa]=(\lambda,\lambda^{*}\otimes\det^{-\kappa})^{\sharp}\otimes\nu^{\kappa}$. Explicitly,
\[ \lambda^{\sharp}[\kappa]=\left((\lambda_{\tau,1},\dots,\lambda_{\tau,n},-\lambda_{\tau,n}-\kappa,\dots,-\lambda_{\tau,1}-\kappa)_{\tau\in\Phi};0\right).\]

For any pair of integers $(m,k)$, we let $\mathcal{E}_{m,\kappa}$ be the fully decomposed automorphic line bundle over $S^{(2)}_{\C}$ corresponding to the one-dimensional irreducible representation of $K_{x^{(2)},\C}$ given by
\[ ((g_{\tau},g_{\tau}')_{\tau\in\Phi};z)\mapsto\prod_{\tau\in\Phi}\det(g_{\tau})^{-m-\kappa}\det(g_{\tau}')^{m}.\]
It is easy to see that this line bundle $\mathcal{E}_{m,\kappa}$ has a canonical model over $\Q$. Its highest weight is parametrized by
\[ \left((-m-\kappa,\dots,-m-\kappa,m,\dots,m)_{\tau\in\Phi};0\right).\]
Recall that $\Lambda$ was defined in the previous subsection. We then obtain an element $\Lambda^{\sharp}[\kappa]\in\Lambda_{c,x^{\sharp}}^{+}$ as above. The corresponding irreducible representation of $K_{x^{\sharp},\C}$ defines an automorphic vector bundle $\mathcal{E}_{\Lambda^{\sharp}[\kappa]}$ over $S^{\sharp}_{\C}$. Its conjugate ${}^{\sigma}\mathcal{E}_{\Lambda^{\sharp}[\kappa]}$, as an automorphic vector bundle over ${}^{\sigma}S_{\C}^{\sharp}$, can be identified with $\mathcal{E}_{{}^{\sigma}\Lambda^{\sharp}[\kappa]}$.

\begin{rem} We correct here a simple misprint of \cite{guerbperiods}, Section 4.5, where an element denoted by $\Lambda^{\sharp}(\ell)$ was used, with $\ell=n\sum_{\tau\in\Phi}m_{\tau}-m_{\bar\tau}$. The correct element to use is $\Lambda^{\sharp}(0)$ (that is, with $\ell=0$). Indeed, the only purpose of $\ell$ was to make sure that the parameter $(\mu+\mu(\eta),\mu^{\vee}-\mu(\eta))^{\sharp}$ equals the Serre dual of $\Lambda^{\sharp}(\ell)$. The computation of these parameters actually shows that the last integer, corresponding to the similitude factor, must be $0$ instead of $\ell$ in both cases, so there is no need to introduce the integer $\ell$, which has no influence on the rest of the proof. Also, note that the $\Lambda^{\sharp}(0)$ of \cite{guerbperiods} is what we call $\Lambda^{\sharp}[0]$ here. In this paper we give a slightly more general version of the results for any integer $\kappa$.
\end{rem}

Let $m\in\Z$ satisfy the inequalities
\begin{equation}\label{mainineq} \frac{n-\kappa}{2}\leq m\leq\min\{-a_{\tau,s_{\tau}+1}+s_{\tau}+m_{\tau}-m_{\bar\tau}-\kappa,a_{\tau,s_{\tau}}+r_{\tau}+m_{\bar\tau}-m_{\tau}\}_{\tau\in\Phi}.\end{equation}
By Proposition 4.2.1 of \cite{guerbperiods}, there exist non-zero differential operators
\begin{equation}\label{diffop} \Delta_{m,\kappa}=\Delta_{m,\kappa}(\Lambda):\mathcal{E}_{m,\kappa}|_{S_{\C}^{\sharp}}\to\mathcal{E}_{\Lambda^{\sharp}[\kappa]},\end{equation}
which are moreover rational over the relevant reflex fields (all of these are contained in $L^{\Gal}$). In {\em op. cit.}, $\kappa$ was taken to be zero, but the proof for any $\kappa$ is completely similar.

If $\sigma\in\Aut(\C)$, then $m$ also satisfies (\ref{mainineq}) for the conjugate Shimura data, and the corresponding differential operator
\[ {}^{\sigma}\Delta_{m,\kappa}=\Delta_{m,\kappa}({}^{\sigma}\Lambda):\mathcal{E}_{m,\kappa}|_{{}^{\sigma}S_{\C}^{\sharp}}\to\mathcal{E}_{{}^{\sigma}\Lambda^{\sharp}[\kappa]} \]
is the conjugate of (\ref{diffop}) under $\sigma$.

\subsection{Petersson norms and CM periods}\label{subsection: petersson}
We recall now the definition of certain CM periods attached to $\psi$ that appear in our critical value formula. The determinant defines a map $\det:G\to T^{\cm}=\Res_{\cm/\Q}\Gm{\cm}$, and thus we have a morphism $\det\circ x:\mathbb{S}\to(T^{\cm})_{\R}$. The pair $(T^{\cm},\det\circ x)$ is a Shimura datum defining a zero dimensional Shimura variety, and the point $\det\circ x$ is a CM point. Recall that $\Q(\psi)$ is the field generated over $\Q$ by the values of $\psi$ on $\A_{\cm,f}^{\times}$. Also, let $E(\mu)\supset E(G,X)$ be the reflex field of the automorphic vector bundle $\mathcal{E}_{\mu}$ over $S$. Define $E(\psi)=E(\mu)E(T^{\cm},\det\circ x)\Q(\psi)$. The infinity type of $\psi$ can be seen as an algebraic character of $T^{\cm}$, and the corresponding automorphic vector bundle $\mathcal{E}_{\psi}$ has a canonical model over $E(\psi)$. Note that $E(T^{\cm},\det\circ x)\supset E(G,X)$.

Attached to the CM point $\det\circ x$ there is a CM period
\[ p(\psi;\det\circ x)\in\C^{\times}, \]
defined in \cite{harriskudla} (see also \cite{harrisunitary}). For every $\sigma\in\Aut(\C)$, the conjugate Shimura datum is canonically identified with $\left(T^{\cm},\det\circ({}^{\sigma}x)\right)$ (this is clear from the definitions), so we can define as well a CM period $p\left({}^{\sigma}\psi;\det\circ({}^{\sigma}x)\right)\in\C^{\times}$. If $\sigma\in\Aut(\C/E(\psi))$, then this coincides with $p(\psi;\det\circ x)$, and this allows us to define $p\left({}^{\rho}\psi;\det\circ({}^{\rho}x)\right)$ for any $\rho\in J_{E(\psi)}$ by extending $\rho$ to an element of $\Aut(\C)$. We let
\[ p^{*}(\psi;\det\circ x)=\left(p\left({}^{\rho}\psi;\det\circ({}^{\rho}x)\right)\right)_{\rho\in J_{E(\psi)}},\]
viewed as an element of $(E(\psi)\otimes\C)^{\times}$. We also define
\[ P(\psi)=P(\psi;x)=p(\psi;\det\circ x)p(\psi^{-1};\det\circ\bar x) \]
and
\[ P^{*}(\psi)=P^{*}(\psi;x)=p^{*}(\psi;\det\circ x)p^{*}(\psi^{-1};\det\circ\bar x).\]
Note that this depends on the choice of the CM point $x$, but we will ignore $x$ for simplicity of notation. If $\alpha$ is another algebraic Hecke character of $\cm$, we let
\[ P^{*}(\psi;\alpha)=P^{*}(\psi;\alpha;x)=p^{*}(\psi;\det\circ x)p^{*}(\psi^{-1}\alpha^{-1};\det\circ\bar x)\in E(\psi,\alpha)\otimes\C,\]
where $E(\psi,\alpha)=E(\psi)\Q(\alpha)$.

As in \cite{guerbperiods}, 3.10, we let $s_{\psi}$ be an automorphic form that contributes to $H^{0}_{!}(S(\det\circ x)_{\C},\mathcal{E}_{\psi})$, which is rational over $E(\psi)$. Similarly, we let $f_{0}$ be an automorphic form in $\pi$, contributing to $H^{d}_{!}(S_{\C},\mathcal{E}_{\mu})$, rational over $E(\mu)$. We can then form an automorphic form $f=f_{0}\otimes s_{\psi}$ on $\pi\otimes\psi$, and a corresponding non-zero $G(\A_{f})$-equivariant map $\gamma:\pi_{f,0}\otimes E(\psi)\to H^{d}_{!}(S_{E(\psi)},\mathcal{E})$, where $\mathcal{E}$ is the automorphic vector bundle over $S_{\C}$ obtained by pulling back $\mathcal{E}_{\psi}$ and taking the tensor product with $\mathcal{E}_{\mu}$. Concretely, $\mathcal{E}$ is attached to the irreducible representation of $K_{x,\C}$ whose highest weight is $\mu+\mu(\psi)$, where
\[ \mu(\psi)=\left((m_{\tau}-m_{\bar\tau},\dots,m_{\tau}-m_{\bar\tau})_{\tau\in\Phi};n\sum_{\tau\in\Phi}m_{\bar\tau}\right).\]
(see \cite{guerbperiods}, 4.5).

We now let $\alpha$ be another algebraic Hecke character of $\cm$, whose infinity type is given at each place $\tau\in\Phi$ by an integer $-\kappa$ (the same for all $\tau\in\Phi$), and at each place $\tau\not\in\Phi$ by $0$. As similar construction as above, using $\pi^{\vee}\otimes\alpha^{-1}$ and $\psi^{-1}$ instead of $\pi$ and $\psi$, gives rise to elements $s_{\psi^{-1}}$, $f_{0}'$ and $f'$, which in turn are associated with a map $\gamma'$ to coherent cohomology in degree $d$ of the conjugate Shimura variety ${}^{c}S_{\C}$. See \cite{guerbperiods}, 4.5, for details. The maps $\gamma$ and $\gamma'$ define via cup product and pullback to $S_{\C}^{\sharp}\hookrightarrow S_{\C}\times{}^{c}S_{\C}$, an element $(\gamma,\gamma')^{\sharp}$ that contributes to
\[ H^{2d}_{!}(S_{\C}^{\sharp},\mathcal{E}_{(\mu+\mu(\psi),\mu^{\vee}-\mu(\psi)-\mu(\alpha))^{\sharp}}).\]
Note that $\mathcal{E}_{(\mu+\mu(\psi),\mu^{\vee}-\mu(\psi)-\mu(\alpha))^{\sharp}}$ is isomorphic to the Serre dual $\mathcal{E}_{\Lambda^{\sharp}[\kappa]}'$ of $\mathcal{E}_{\Lambda^{\sharp}[\kappa]}$.

For $\sigma\in\Aut(\C)$, we can conjugate $s_{\psi}$, $s_{\psi^{-1}}$, $f_{0}$ and $f_{0}'$ (and hence $f$ and $f'$) to obtain automorphic forms ${}^{\sigma}f\in{}^{\sigma}\pi\otimes{}^{\sigma}\psi$ and ${}^{\sigma}f'\in{}^{\sigma}\pi^{\vee}\otimes{}^{\sigma}\alpha^{-1}\otimes{}^{\sigma}\psi^{-1}$. These are also associated with ${}^{\sigma}G(\A_{f})$-equivariant maps ${}^{\sigma}\gamma$ and ${}^{\sigma}\gamma'$, and the same procedure as above gives rise to an element
\[ {}^{\sigma}(\gamma,\gamma')^{\sharp}=({}^{\sigma}\gamma,{}^{\sigma}\gamma')^{\sharp} \]
that contributes to
\[ H^{2d}_{!}({}^{\sigma}S_{\C}^{\sharp},\mathcal{E}_{{}^{\sigma}\Lambda^{\sharp}[\kappa]}'). \]

We define
\[ Q^{\Pet}(f_{0})=\int_{Z(\A)G(\Q)G(\A)}f_{0}(g)\bar{f_{0}}(g)\|\nu(g)\|^{2a_{0}}dg,\]
and we define $Q^{\Pet}({}^{\sigma}f_{0})$ for $\sigma\in\Aut(\C)$ in a similar way. We let $E(\pi)=E(\mu)E_{0}(\pi)$, and $E(\pi,\psi)=E(\pi)E(\psi)$. By (\ref{hypomult}), we can define $Q^{\Pet}(\pi)=Q^{\Pet}(f_{0})$ uniquely up to multiples by $E(\pi)$. If $\sigma$ fixes $E(\pi)$ (in particular, if $\sigma$ fixes $E(\pi,\psi)$), then $Q^{\Pet}({}^{\sigma}f_{0})=Q^{\Pet}(f_{0})$, and hence we can define
\[ Q^{\Pet,*}(\pi)=\left(Q^{\Pet}({}^{\rho}\pi)\right)_{\rho\in J_{E(\pi,\psi)}}\in E(\pi,\psi)\otimes\C.\]

We also let
\[ (f,f')=\int_{Z(\A)G(\Q)\backslash G(\A)}f(g)f'(g)\alpha\left(\det(g)\right)dg,\]
and get in a similar fashion an element
\[ (f,f')^{*}=\left(({}^{\rho}f,^{\rho}f')_{{}^{\rho}G}\right)_{\rho\in J_{E(\pi,\psi,\alpha)}}\in E(\pi,\psi,\alpha)\otimes\C,\]
where $E(\pi,\psi,\alpha)=E(\pi,\psi)\Q(\alpha)$.

\begin{lemma}\label{lemma petersson} Keep the notation and assumptions as above. Then
\[ (f,f')^{*}\sim_{E(\pi,\psi,\alpha)\otimes\cm^{\Gal}}(2\pi i)^{2a_{0}}Q^{\Pet,*}(\pi)P^{*}(\psi;\alpha)^{-1}.\]
\begin{proof} This is completely similar to the computations in Section 2.9 of \cite{harriscrelle}.
	\end{proof}
\end{lemma}

\begin{rem} The $L$-function $L^{*,\mot,S}(s,\pi\otimes\psi,\St,\alpha)$ can be seen as valued in $E(\pi,\psi,\alpha)\otimes\C$.
\end{rem}

\subsection{Eisenstein series and zeta integrals}
Let $\alpha$ be an algebraic Hecke character of $\cm$ as above. For $s\in\C$, let $I(s,\alpha)$ be the induced representation
\begin{eqnarray}
 &I(s,\alpha)=&\nonumber \\
& \nonumber \{f:G^{(2)}(\A)\to\C:f(pg)=\delta_{GP,\A}(p,\alpha,s)f(g),\hspace{1mm}g\in G^{(2)}(\A),\hspace{1mm}p\in GP(\A)\},&
 \end{eqnarray}
where $\delta_{GP,\A}(p,\alpha,s)=\alpha\left(\det(A(p))\right)\|N_{L/K}\det A(p)\|_{\A_{K}}^{\frac{n}{2}+s}\|\nu(p)\|_{\A_{K}}^{-\frac{n^{2}}{2}-ns}$. The local inductions $I(s,\alpha)_{v}$ and finite and archimedean inductions $I(s,\alpha)_{f}$ and $I(s,\alpha)_{\infty}$ are defined similarly. A section of $I(s,\alpha)$ is a function $\phi(\cdot,\cdot)$, that to each $s\in\C$ assigns an element $\phi(\cdot,s)\in I(s,\alpha)$, with a certain continuity property. Local sections are defined similarly. For $\operatorname{Re}(s)\gg0$, we can defined the Eisenstein series
\[ E_{\phi,s}(g)=\sum_{\sigma\in GP(\Q)\backslash G^{(2)}(\Q)}\phi(\sigma g,s),\]
which converges absolutely to an automorphic form on $G^{(2)}(\A)$. This extends meromorphically to a function of $s\in\C$.

%

From now on, fix $m>n-\frac{\kappa}{2}$ an integer satisfying (\ref{mainineq}). Let $f\in\pi\otimes\psi$ and $f'\in\pi^{\vee}\otimes\alpha^{-1}\otimes\psi^{-1}$ as above. For any section $\phi$ of $I(s,\alpha)$, we define the modified Piatetski-Shapiro-Rallis zeta integral to be
\[ Z(s,f,f',\phi)=\int_{Z^{\sharp}(\A)G^{\sharp}(\Q)\backslash G^{\sharp}(\A)}E_{\phi,s}(i(g,g'))f(g)f'(g')dgdg',\]
where $Z^{\sharp}$ is the center of $G^{\sharp}$. Suppose moreover that $f$, $f'$ and $\phi$ are factorizable as $\otimes_{v}'f_{v}$, $\otimes'_{v}f'_{v}\otimes\alpha_{v}^{-1}$ and $\prod'_{v}\phi_{v}$. Note that we are taking $f'_{v}\in\pi^{\vee}_{v}\otimes\psi_{v}^{-1}$, with $f'_{v}\otimes\alpha^{-1}_{v}$ the function sending $g$ to $f'_{v}(g)\alpha^{-1}\left(\det(g)\right)$. At almost all places $v$, $\pi_{v}\otimes\psi_{v}$ is unramified and $f_{v}$ and $f'_{v}$ are normalized spherical vectors of $\pi_{v}\otimes\psi_{v}$ and $\pi^{\vee}_{v}\otimes\psi^{-1}_{v}$ respectively, with the local pairing $(f_{v},f'_{v})=1$. Define the local zeta integrals as
\[ Z_{v}(s,f,f',\phi)=\int_{U_{v}}\phi_{v}(i(h_{v},1),s)c_{f,f',v}(h_{v})dh_{v},\]
where $U_{v}$ is the local unitary group at the place $v$ for $V$, and
\[ c_{f,f',v}(h_{v})=(f_{v},f'_{v})^{-1}(\pi_{v}(h_{v})f_{v},f'_{v}) \]
is a normalized matrix coefficient for $\pi_{v}$. We let $S$ be a big enough set of primes of $K$ containing the archimedean primes (in practice we take $S$ to be the set consisting of the archimedean places $S_{\infty}$, the places at which $G$ is not quasi-split and the places $v$ where $\pi_{v}$ is ramified or $f_v$ or $f'_{v}$ is not a standard spherical vector). Write $S=S_{f}\cup S_{\infty}$, and let
\[ Z_{f}(s,f,f',\phi)=\prod_{v\in S_{f}}Z_{v}(s,f,f',\phi) \]
and
\[ Z_{\infty}(s,f,f',\phi)=\prod_{v\in S_{\infty}}Z_{v}(s,f,f',\phi).\]

We can conjugate sections $\phi_{f}$ by an element $\sigma\in\Aut(\C)$ as in the discussion before Lemma 6.2.7 of \cite{harrisunitary}.

\begin{lemma}\label{lemma zetaf} There exists a finite section $\phi_{f}(\cdot,s)\in I(s,\alpha)_{f}$ with \\ $\phi_{f}\left(\cdot,m-\frac{n}{2}\right)$ taking values in $\Q(\alpha)$ such that
\[ Z_{f}\left(m-\frac{n}{2},f,f',\phi_{f}\right)\neq0.\]
Moreover, for any $\sigma\in\Aut(\C)$, we have
\[ \sigma\left(Z_{f}\left(m-\frac{n}{2},f,f',\phi_{f}\right)\right)=Z_{f}\left(m-\frac{n}{2},{}^{\sigma}f,{}^{\sigma}f',{}^{\sigma}\phi_{f}\right).\]
In particular,
\[ Z_{f}\left(m-\frac{n}{2},f,f',\phi_{f}\right)\in E(\pi,\psi,\alpha).\]
\begin{proof} The existence of $\phi_{f}$ with the first property follows as in Lemma 4.5.2 of \cite{guerbperiods} or Lemma 3.5.7 of \cite{harriscrelle} (see also the proof of Theorem 4.3 of \cite{siegelweil}). The description of the action of $\sigma$ follows from Lemma 6.2.7 of \cite{harrisunitary}.
\end{proof}
\end{lemma}

From now on, fix $\phi_{f}$ as in Lemma \ref{lemma zetaf}. We consider the element $G(\alpha)\in E(\pi,\psi,\alpha)\otimes\C$ defined in Subsection \ref{subsection: algebraic Hecke}, and denote its $\rho$-component by $G(\alpha)_{\rho}$, for $\rho:E(\pi,\psi,\alpha)\hookrightarrow\C$. If $\sigma\in\Aut(\C)$, we let $G(\alpha)_{\sigma}=G(\alpha)_{\rho}$, where $\rho$ is the restriction of $\sigma$ to $E(\pi,\psi,\alpha)$. Define a section $\varphi_{m,\kappa,\sigma}$ of $I\left(s+m-\frac{n}{2},{}^{\sigma}\alpha\right)$ by
\begin{eqnarray}
 &\varphi_{m,\kappa,\sigma}(g,s)=&\nonumber\\
 &\nonumber \mathbb{J}_{m,\kappa}\left(g,s+m-\frac{n}{2}\right)\otimes(2\pi i)^{[\tr:\Q](m+\kappa)n}G(\alpha)_{\sigma}^{n}({}^{\sigma}\phi_{f})(g,s+m-\frac{n}{2}).&
 \end{eqnarray}
The element $\mathbb{J}_{m,\kappa}$ is defined in \cite{cohomologicalII}, (1.2.7) (with a misprint correction, see \cite{guerbperiods}, 4.3). The Eisenstein series $E_{m,\kappa}=E_{m,\kappa,1}=E_{\varphi_{m,\kappa,1}}$ has no pole at $s=0$ (see for example (1.2.5) of \cite{siegelweil}, where $\chi=\alpha\|N_{\cm/\tr}\|^{-\kappa/2}$), and thus this defines an automorphic form, also denoted by $E_{m,\kappa}$, on $G^{(2)}(\A)$, which can be seen as an element of $H^{0}(S_{\C}^{(2)},\mathcal{E}_{m,\kappa}^{\can})$ (see \cite{guerbperiods}, 4.3). Using the differential operators $\Delta_{m,\kappa}$, as explained in {\em op. cit.}, we can define sections
\[ \tilde\varphi_{m,\kappa,\sigma}=\Delta_{m,\kappa}\varphi_{m,\kappa,\sigma}\]
for $\sigma\in\Aut(\C)$, and a corresponding Eisenstein series
\[ \tilde E_{m,\kappa}=E_{\tilde\varphi_{m,\kappa,1}}.\]
Then $\tilde E_{m,\kappa}$ equals $\Delta_{m,\kappa}E_{m,\kappa}$ when restricted to $G^{\sharp}(\A)$.

\begin{prop}\label{eisenstein rational} The Eisenstein series $E_{m,\kappa}$ and $\tilde E_{m,\kappa}$ are rational over $\Q(\alpha)$ with respect to the canonical models of $S^{(2)}_{\C}$ and $\mathcal{E}_{m,\kappa}$. Moreover, for any $\sigma\in\Aut(\C)$,
\[ {}^{\sigma}E_{m,\kappa}=E_{\varphi_{m,\kappa,\sigma}},\]
and a similar equation holds for $\tilde E_{m,\kappa}$.
\begin{proof} This follows by combining the ideas of Lemma 3.3.5.3 of \cite{harriscrelle} and Proposition 4.3.1 of \cite{guerbperiods}. Namely, in the latter, we just need to note that the character $\tilde\lambda$ is now given by
\[ \tilde\lambda(p)=\left(N_{\cm/\Q}\det\left(A(p)\right)\right)^{-m}\nu(p)^{-[\tr:\Q]nm}t_{\alpha}\left(\det\left(A(p)\right)\right)^{-1},\]
where $t_{\alpha}$ is the algebraic character of $\Res_{\cm/\Q}\Gm{\cm}$, defined over $\Q(\alpha)$, inverse of the infinity type of $\alpha$, so that the restriction to $\Sh(\Gm{\Q},N)$ is the Tate automorphic vector bundle $\Q\left(-[\tr:\Q]n(m+\kappa)\right)$.
\end{proof}
\end{prop}

We define
\[ Z^{*}\left(m-\frac{n}{2},f,f',\tilde\varphi_{m,\kappa}\right)=\left(Z\left(m-\frac{n}{2},{}^{\sigma}f,{}^{\sigma}f',\tilde\varphi_{m,\kappa,\sigma}\right)\right)_{\sigma\in\Aut(\C)}. \]
The elements of this family only depend on the restrictions of elements $\sigma\in\Aut(\C)$ to $E(\pi,\psi,\alpha)$, and hence we can consider
\[ Z^{*}\left(m-\frac{n}{2},f,f',\tilde\varphi_{m,\kappa}\right)=\left(Z\left(m-\frac{n}{2},{}^{\rho}f,{}^{\rho}f',\tilde\varphi_{m,\kappa,\rho}\right)\right)_{\rho\in J_{E(\pi,\psi,\alpha)}}\] 
as an element of $E(\pi,\psi,\alpha)\otimes\C$. We can also define
\[ Z^{*}_{\infty}\left(m-\frac{n}{2},f,f',\tilde\varphi_{m,\kappa}\right)=\left(Z_{\infty}\left(m-\frac{n}{2},{}^{\rho}f,{}^{\rho}f',\tilde\varphi_{m,\kappa,\rho}\right)\right)_{\rho\in J_{E(\pi,\psi,\alpha)}},\]
which is an element of $E(\pi,\psi,\alpha)\otimes\C$. Note that the archimedean part of $\tilde\varphi_{m,\kappa,\rho}$ is independent of $\rho$, and hence so are the archimedean zeta integrals. Finally, we can define
\[ Z^{*}_{f}\left(m-\frac{n}{2},f,f',\tilde\varphi_{m,\kappa}\right)=\left(Z_{f}\left(m-\frac{n}{2},{}^{\rho}f,{}^{\rho}f',\tilde\varphi_{m,\kappa,\rho}\right)\right)_{\rho\in J_{E(\pi,\psi,\alpha)}} \]
and $Z^{*}_{f}\left(m-\frac{n}{2},f,f',\phi_{f}\right)$.


\begin{lemma}\label{lemma zeta} Let the notation and assumptions be as above. Then
\[ Z^{*}\left(m-\frac{n}{2},f,f',\tilde\varphi_{m,\kappa}\right)\in E(\pi,\psi,\alpha).\]
\begin{proof} Let
\begin{equation}\label{duality} \mathcal{L}_{m,\kappa}:H^{2d}(S_{\C}^{\sharp},\mathcal{E}_{\Lambda^{\sharp}[\kappa]}')\to\C \end{equation}
be the map defined by pairing with $\Delta_{m,\kappa}E_{m,\kappa}$ via Serre duality. Then, as in Lemma 4.5.3 of \cite{guerbperiods}, we have
\[ Z\left(m-\frac{n}{2},f,f',\tilde\varphi_{m,\kappa,1}\right)=\mathcal{L}_{m,\kappa}\left((\gamma,\gamma')^{\sharp}\right).\]
Let $\sigma\in\Aut(\C)$. The conjugate of (\ref{duality}) by $\sigma$ is now
\[ \mathcal{L}_{m,\kappa}:H^{2d}({}^{\sigma}S_{\C}^{\sharp},\mathcal{E}_{{}^{\sigma}\Lambda^{\sharp}[\kappa]}')\to\C, \]
which is given by cup product with ${}^{\sigma}\Delta_{m,\kappa}E_{m,\kappa}$ via Serre duality. It then follows that
\[ \sigma\left(Z\left(m-\frac{n}{2},f,f',\tilde\varphi_{m,\kappa,1}\right)\right)=\mathcal{L}_{m,\kappa}\left({}^{\sigma}(\gamma,\gamma')^{\sharp}\right),\]
which equals
\[ Z\left(m-\frac{n}{2},{}^{\sigma}f,{}^{\sigma}f',{}^{\sigma}\tilde\varphi_{m,\kappa,\sigma}\right) \]
by Proposition \ref{eisenstein rational} and the same reasoning as above. This finishes the proof of the lemma.
\end{proof}
\end{lemma}

The main formula for the doubling method, proved by Li in \cite{li}, says that
\begin{equation}\label{main formula}
d^{S}\left(s-\frac{n}{2},\alpha\right)Z\left(s-\frac{n}{2},f,f',\phi\right)=  \end{equation}
\[(f,f')\prod_{v\in S}Z_{v}\left(s-\frac{n}{2},f,f',\phi\right)L^{\mot,S}(s,\pi\otimes\psi,\St,\alpha) \]
for any section $\phi$. Here
\[ d^{S}(s,\alpha)=\prod_{j=0}^{n-1}L^{S}(2s+n-j,\alpha|_{\A_{\tr}^{\times}}\varepsilon_{\cm}^{j}), \]
 where $\varepsilon_{\cm}$ is the quadratic character associated with the quadratic extension $\cm/\tr$. We can write
\[ \alpha|_{\A_{\tr}^{\times}}=\alpha_{0}\|\cdot\|_{\A_{\tr}^{\times}}^{\kappa}\]
with $\alpha_{0}$ of finite order, so that
\[ d^{S}(s,\alpha)=\prod_{j=0}^{n-1}L^{S}(2s+n-j+\kappa,\alpha_{0}\varepsilon_{\cm}^{j}).\]
We let
\[ d^{*}(s,\alpha)=\prod_{j=0}^{n-1}L^{*}(2s+n-j+\kappa,\alpha_{0}\varepsilon_{\cm}^{j})\in\Q(\alpha_{0})\otimes\C,\]
and we define similarly $d^{*,S}(s,\alpha)$ by removing the local factors at primes of $S$. We can deduce from (\ref{main formula}) that
\begin{equation}\label{main formula star}
d^{*,S}\left(m-\frac{n}{2},\alpha\right)Z^{*}\left(m-\frac{n}{2},f,f',\tilde\varphi_{m,\kappa}\right)(f,f')^{*,-1}=  \end{equation}
\[Z_{f}^{*}\left(m-\frac{n}{2},f,f',\tilde\varphi_{m,\kappa}\right)Z_{\infty}^{*}\left(m-\frac{n}{2},f,f',\tilde\varphi_{m,\kappa}\right)L^{*,\mot,S}(m,\pi\otimes\psi,\St,\alpha) \]

\begin{lemma}\label{lemma d} We have
\[ d^{*,S}\left(m-\frac{n}{2},\alpha\right)\sim_{\Q(\alpha_{0})}(2\pi i)^{[\tr:\Q]\left((2m+\kappa)n-\frac{n(n-1)}{2}\right)}D_{\tr}^{\lfloor\frac{n+1}{2}\rfloor/2}\delta[\varepsilon_{\cm}]^{\lfloor\frac{n}{2}\rfloor}G(\alpha)^{n}. \]
\begin{proof} First, suppose that $0\leq j\leq n-1$ is even. Note that $2m-j+\kappa$ is even and positive, and hence is a critical integer of $L^{*}(s,\alpha_{0})$, because the motive $\Res_{\tr/\Q}[\alpha_{0}]$ is purely of type $(0,0)$ and the Frobenius involution acts as $(-1)^{\kappa}$. Since Deligne's conjecture is known for $\Res_{\tr/\Q}[\alpha_{0}]$, we get
\[ L^{*}(2m-j+\kappa,\alpha_{0})\sim_{\Q(\alpha_{0})}(2\pi i)^{[\tr:\Q](2m-j+\kappa)}c^{\pm}[\alpha_{0}],\]
where $\pm=(-1)^{2m-j+\kappa}=(-1)^{\kappa}$. Here we are writing 
\[c^{\pm}[\alpha_{0}]=c^{\pm}\left(\Res_{\tr/\Q}[\alpha_{0}]\right).\] Similarly, if $0\leq j\leq n-1$ is odd, then $2m-j+\kappa$ is a critical integer for the motive $\Res_{\tr/\Q}[\alpha_{0}\varepsilon_{\cm}]$ and
\[ L^{*}(2m-j+\kappa,\alpha_{0}\varepsilon_{\cm})\sim_{\Q(\alpha_{0})}(2\pi i)^{[\tr:\Q](2m-j+\kappa)}c^{\mp}[\alpha_{0}\varepsilon_{\cm}].\]
We know use Remark 2.2.1 of \cite{guerbperiods}, together with Proposition 2.2 of \cite{yoshida}, to get
\[ c^{\pm}[\alpha_{0}]\sim_{\Q(\alpha_{0})}\delta[\alpha_{0}] \]
and
\[ c^{\mp}[\alpha_{0}\varepsilon_{\cm}]\sim_{\Q(\alpha_{0})}\delta[\alpha_{0}]\delta[\varepsilon_{\cm}]D_{\tr}^{-1/2}. \]
The lemma follows from these computations, combined with (\ref{delta0 G}) and the fact that
\[ d^{*,S}\left(m-\frac{n}{2},\alpha\right)\sim_{\Q(\alpha)}d^{*}\left(m-\frac{n}{2},\alpha\right).\]
\end{proof}
\end{lemma}

\subsection{Modified periods}
 A theorem of Garrett (\cite{garrett}) says that the archimedean zeta integral $Z_{\infty}\left(m-\frac{n}{2},f,f',\tilde\varphi_{m}\right)$ is non-zero (and, moreover, belongs to $\cm^{\Gal}$), and we define the modified (Petersson) period
 \[ \mathfrak{Q}^{\Pet}(m;\pi,\psi,\alpha)=\frac{(f,f')^{-1}}{Z_{\infty}\left(m-\frac{n}{2},f,f',\tilde\varphi_{m,\kappa}\right)}.\]
It follows from Lemma \ref{lemma petersson} that
 \[ \mathfrak{Q}^{\Pet}(m;\pi,\psi,\alpha)\sim_{E(\pi,\psi,\alpha)\otimes\cm^{\Gal}}(2\pi i)^{-2a_{0}}Q^{\Pet}(\pi)^{-1}P(\psi;\alpha).\]
More generally, we can define $\mathfrak{Q}^{\Pet,*}(m;\pi,\psi,\alpha)\in\left(E(\pi,\psi,\alpha)\otimes\C\right)^{\times}$ as
\begin{eqnarray}\nonumber
\mathfrak{Q}^{\Pet,*}(m;\pi,\psi,\alpha)&=&\left(\mathfrak{Q}^{\Pet}(m;\pi^{\rho},\psi^{\rho},\alpha^{\rho})\right)_{\rho\in J_{E(\pi,\psi,\alpha)}}\\
\nonumber &=&\frac{(f,f')^{*,-1}}{Z_{\infty}^{*}\left(m-\frac{n}{2},f,f',\tilde\varphi_{m,\kappa}\right)}.
\end{eqnarray}

The doubling zeta integral agains the Eisenstein series $\tilde E_{m,\kappa}$ defines a bilinear form
\[ B^{\alpha}:H^{d}_{!}(S_{\C},\mathcal{E})[\pi\otimes\psi]\times H^{d}_{!}(\bar S_{\C},\mathcal{E}^{*})[\pi^{\vee}\otimes\psi^{-1}\otimes\alpha^{-1}]\to\C,\]
which is moreover rational over $E(\pi,\psi,\alpha)$. Here $\mathcal{E}$ and $\mathcal{E}^{*}$ are the automorphic vector bundles determined by $\pi$ and $\pi^{\vee}\otimes\alpha^{-1}$. In particular, with our choice of $f$ and $f'$, we have that $B^{\alpha}(f,f')\in E(\pi,\psi,\alpha)$. Moreover, for any $\sigma\in\Aut(\C)$,
\begin{equation}\label{Galois equivariance Bzeta} \sigma\left(B^{\alpha}(f,f')\right)=B^{{}^{\sigma}\alpha}\left({}^{\sigma}f,{}^{\sigma}f'\right).\end{equation}
By our multiplicity assumptions, any other bilinear form, such as the Petersson integral, must be a scalar multiple of $B^{\alpha}$. In particular, there exists an element $Q(\pi,\psi,\alpha)\in\C^{\times}$ such that
\[ B^{\alpha}(f,f')=Q(\pi,\psi,\alpha)(f,f').\]
We can define $Q^{*}(\pi,\psi,\alpha)\in(E(\pi,\psi,\alpha)\otimes\C)^{\times}$ by taking
\[ Q^{*}(\pi,\psi,\alpha)=\left(Q({}^{\rho}\pi,{}^{\rho}\psi,{}^{\rho}\alpha)\right)_{\rho\in J_{E(\pi,\psi,\alpha)}}.\]
By (\ref{Galois equivariance Bzeta}), we have that
\begin{equation}\label{Pet = Q} (f,f')^{*,-1}\sim_{E(\pi,\psi,\alpha)}Q^{*}(\pi,\psi,\alpha).\end{equation}

\begin{rem} In the above computations, $Q(\pi,\psi,\alpha)$ depends a priori on the Eisenstein series $\tilde E_{m,\kappa}$, and hence on the integer $m$. However, (\ref{Pet = Q}) shows that, up to multiplication by an element in $E(\pi,\psi,\alpha)\subset E(\pi,\psi,\alpha)\otimes\C$, $Q^{*}(\pi,\psi,\alpha)$ does not depend on $m$ or the Eisenstein series. 
\end{rem}

We define
\[ \mathfrak{Q}^{*}(m;\pi,\psi,\alpha)=\frac{Q^{*}(\pi,\psi,\alpha)}{Z_{\infty}^{*}\left(m-\frac{n}{2},f,f',\tilde\varphi_{m,\kappa}\right)}.\]
We have that
\[ \mathfrak{Q}^{\Pet,*}(m;\pi,\psi,\alpha)\sim_{E(\pi,\psi,\alpha)}\mathfrak{Q}^{*}(m;\pi,\psi,\alpha).\]

\subsection{The main theorem} Before stating our main theorem, we recall all the hypothesis and assumptions that we have made so far. Thus, $\cm/\tr$ is a CM extension, $\Phi$ is a CM type, and $\pi$ is an automorphic representation of $G(\A)$, satisfying hypotheses \ref{hypotheses} for a parameter \[ \mu=\left((a_{\tau,1},\dots,a_{\tau,n})_{\tau\in\Phi};a_{0}\right)\] (recall as well the assumption that $\pi$ can be conjugated to ${}^{\sigma}\pi$ with the desired properties). We also assume that $\pi^{\vee}\cong\pi\otimes\|\nu\|^{2a_{0}}$, $\pi$ contributes to antiholomorphic cohomology and satisfies the multiplicity assumption (\ref{hypomult}).

We also have algebraic Hecke characters $\psi$ and $\alpha$ of $\cm$. The infinity type of $\psi$ is $(m_{\tau})_{\tau\in J_{\cm}}$, and that of $\alpha$ is given by an integer $\kappa$ at places of $\Phi$, and by $0$ at places outside $\Phi$. We define the number field $E(\pi,\psi,\alpha)$ as in Subsection \ref{subsection: petersson}.


\begin{thm}\label{main theorem} Keep the notation and assumptions as above, and let $m>n-\frac{\kappa}{2}$ be an integer satisfying (\ref{mainineq}). Then
	\[ L^{*,\mot,S}\left(m,\pi\otimes\psi,\St,\alpha\right)\sim_{E(\pi,\psi,\alpha)}\]\[(2\pi i)^{[\tr:\Q](mn-n(n-1)/2)}D_{\tr}^{\lfloor\frac{n+1}{2}\rfloor/2}\delta[\varepsilon_{\cm}]^{\lfloor\frac{n}{2}\rfloor}\mathfrak{Q}^{*}(m;\pi,\psi,\alpha).\]
\begin{proof} We use formula (\ref{main formula star}). By Lemma \ref{lemma zetaf}, we have that
\[ Z_{f}^{*}\left(m-\frac{n}{2},f,f',\tilde\varphi_{m,\kappa}\right)\sim_{E(\pi,\psi,\alpha)}(2\pi i)^{[\tr:\Q](m+\kappa)n}G(\alpha)^{n}.\]
Also, Lemma \ref{lemma zeta} says that
\[ Z^{*}\left(m-\frac{n}{2},f,f',\tilde\varphi_{m,\kappa}\right)\sim_{E(\pi,\psi,\alpha)}1.\]
The formula in the theorem follows immediately from these, Lemma \ref{lemma d} and the definition of $\mathfrak{Q}(m;\pi,\psi,\alpha)$.

\end{proof}
\end{thm}

\begin{rem} The condition $m>n-\kappa/{2}$ guarantees that we are in the range of absolute convergence for the Eisenstein series involved in the method. In a series of papers (\cite{cohomologicalI}, \cite{cohomologicalII}, \cite{siegelweil}), Harris extended the formulas for special values and applications to period relations when the base field $\tr$ is $\Q$. One of the main ingredients of Harris's method to extend the results is a careful study of the theta correspondence and its rationality properties. It should be possible to generalize these results to the setup of this paper, namely, to the case of a general CM extension $\cm/\tr$. 

\end{rem}

\subsection{A refinement}
Lemma \ref{lemma petersson} gives us a factorization of $(f,f')$ in terms of periods associated to $\pi$, $\psi$ and $\alpha$ respectively. This will lead to a finer result on the special values as in Theorem $1$ in \cite{guerbperiods}.

Unfortunately, the relation in Lemma \ref{lemma petersson} is only shown under action of $\Gal(\overline{\Q}/\cm^{\Gal})$. We hope to prove a $\Gal(\overline{\Q}/\Q)$ version in the near future.

Before we state the main formula, let us define two more factors which will appear.

\begin{dfn}
\begin{enumerate}
\item Let $j \in \cm$ be a purely imaginary element, i.e., $\overline{j}=-j$ where $\overline{j}$ refers to the complex conjugation of $j$ in the CM field $\cm$. We define 
\[ \Icm=\prod\limits_{\tau\in\Phi}\tau(j).\] 
Its image in $\C^{\times}/\Q^{\times}$ does not depend on the choice of the purely imaginary element $j$ or the CM type $\Phi$.
\item Let $E$ be a number field containing $\cm^{\Gal}$. 
We fix $\rho_{0}$ an element in $J_{E}$. For any $\rho\in J_{E}$, we define a sign $e_{\Phi}(\rho)$ as $(-1)^{\#(\Phi\backslash g\Phi)} $ by taking any $g\in\Aut(\C)$ such that $g\rho_{0}=\rho$. We can see easily that it does not depend on the choice of $g$.

We define $e_{\Phi}=(e_{\Phi}(\rho))_{\rho\in J_{E}}$ as an element of $E\otimes \C$.

\end{enumerate}
\end{dfn}


\begin{coro}\label{special value unitary group}
With the same assumption as in Theorem \ref{main theorem}, we have that
\[ L^{*,\mot,S}\left(m,\pi\otimes\psi,\St,\alpha\right)\sim_{E(\pi,\psi,\alpha);F^{\Gal}}\]\[(2\pi i)^{[\tr:\Q](mn-n(n-1)/2)-2a_{0}}\Icm^{[n/2]} (D_{\tr}^{1/2})^{n}e_{\Phi}^{mn}Q^{*}(\pi)^{-1}P^{*}(\psi,\alpha).\]

\end{coro}

\begin{rem}
 \begin{enumerate}
 \item We identify $(2 \pi i)^{mnd(\tr)} \Icm^{[n/2]} (D_{\tr}^{1/2})^{n}$ with 
 \[ 1\otimes (2 \pi i)^{mnd(\tr)} \Icm^{[n/2]} (D_{\tr}^{1/2})^{n}\]
 as an element in $E(\Pi,\eta)\otimes \C$.

 \item It is not difficult to see that if $g\in\Aut(\C)$ fixes $\cm^{\Gal}$ then it fixes 
 \[ \Icm^{[n/2]} (D_{\tr}^{1/2})^{n}e_{\Phi}^{mn}.\]
 Consequently, it is an element inside $E(\pi,\psi,\alpha)\otimes \cm^{\Gal}$ and hence can be ignored here. We still keep it because it is predicted by Deligne's conjecture if we want a finer result up to $E(\pi,\psi,\alpha)\otimes\Q$, or equivalently, under the action of the full Galois group $\Gal(\overline{\Q}/\Q)$.
 
 \end{enumerate}
 \end{rem}
 
 \begin{proof}
 By the proof of Lemma $2.4.2$ of \cite{guerbperiods} and Proposition $2.2$ of \cite{yoshida} we know that \[\delta[\varepsilon_{\cm}]\sim_{\Q} \Icm D_{\tr}^{1/2}.\]

Moreover, Lemma \ref{lemma petersson} implies that
 \[ \mathfrak{Q}^{*}(m;\pi,\psi,\alpha)\sim_{E(\pi,\psi,\alpha)\otimes\cm^{\Gal}}(2\pi i)^{-2a_{0}}Q^{*}(\pi)^{-1}P^{*}(\psi,\alpha).\]
 
It remains to show that $e_{\Phi}\in E\otimes \cm^{\Gal}$. In fact, let $\rho\in J_{E}$ and $g\in\Aut(\C/\cm^{\Gal})$. We take $h\in\Aut(\C)$ such that $\rho=h \rho_{0}$. By definition $e_{\Phi}(\rho)=(-1)^{\#(\Phi\backslash h\Phi)} $ and $e_{\Phi}(g \rho)=(-1)^{\#(\Phi\backslash gh\Phi)}$. Since $g$ fixes $\cm^{\Gal}$, we have $gh\Phi=h\Phi$ and hence $e_{\Phi}(g \rho)=e_{\Phi}( \rho)$. We conclude that $e_{\Phi}\in E\otimes \cm^{\Gal}$ by Definition-Lemma $1.1$ of \cite{lincomptesrendus}. The corollary then follows from Theorem \ref{main theorem}.
 
 \end{proof}

\begin{rem}
We expect that Lemma \ref{lemma petersson} is true up to factors in \\$E(\pi,\psi,\alpha)$. Moreover, we hope to show that 
\[ Z_{\infty}^{*}\left(m-\frac{n}{2},f,f',\tilde\varphi_{m,\kappa}\right)\sim_{E(\pi,\psi,\alpha)} e_{\Phi}^{mn}.\]
By the method explained in section $9.4$ of \cite{linthesis} and Blasius's proof of Deligne's conjecture for algebraic Hecke characters (\cite{blasiusannals}), we can reduce to show that certain archimedean zeta integral belongs to $\Q$. Garett proved this for particular cases (see \cite{garrett}). We hope to solve this in the future and then the above corollary is true up to $E(\pi,\psi,\alpha)$.
\end{rem}

\section{Applications to general linear groups}

\subsection{Transfer from similitude unitary groups to unitary groups}
\text{}\\

Let $\pi$ be an automorphic representation of $G(V)(\AQ)$. We want to consider the restriction of $\pi$ to $U(V)(\AQ)$.  We sketch the construction of \cite{labesse-schwermer} in our case.

\begin{dfn}
\begin{enumerate}
\item Let $\pi_{1}$ and $\pi_{2}$ be two admissible irreducible representation of $G(V)(\AQ)$. We say that they are $\mathcal{E}$-equivalent if there exists a character $\chi$ of $U(V)(\AQ)\backslash G(V)(\AQ)$ such that $\pi_{1}\cong \pi_{2}\otimes \chi$.
\item Let $\pi_{0}$ be an admissible irreducible representation of $U(V)(\AQ)$ and $g$ be an element in $G(V)(\AQ)$. We define $\pi^{g}$, a new representation on $U(V)(\AQ)$, by $\pi^{g}(x)=\pi(gxg^{-1})$.
\item Let $\pi_{0,1}$ and $\pi_{0,2}$ be two admissible irreducible representation of $U(V)(\AQ)$. We say that they are $\mathcal{L}$-equivalent if there exists $g\in G(V)(\AQ)$ such that $\pi_{0,1}\cong \pi_{0,2}^{g}$.
\end{enumerate}
\end{dfn}

\begin{lemma}
Let $\pi$ be an admissible irreducible automorphic representation of $G(V)(\AQ)$. The restriction of $\pi$ to $U(V)(\AQ)$ is a direct sum of admissible irreducible representations in the same $\mathcal{L}$-equivalence class. This gives a bijection of the $\mathcal{E}$-equivalence classes of admissible irreducible representations of $G(V)(\AQ)$ and the $\mathcal{L}$-equivalence classes of admissible irreducible representations of $G(V)(\AQ)$.

Moreover, if we restrict to the cuspidal spectrum then we get a bijection of equivalence classes of cuspidal representations on both sides.
\end{lemma}

\begin{proof}
The proof is the similar as in Lemma $3.3$ and Proposition $3.5$ of \cite{labesse-schwermer} for the special linear group. More details for unitary groups can be found in section $5$ of \cite{clozelIHES}. We sketch the idea there for the last statement.

We write $S$ for the maximal split central torus of $G$. It is isomorphic to $\mathbb{G}_{m}$. Its intersection with $U$ is then isomorphic to $\mu_{2}\subset \mathbb{G}_{m}$. As in \cite{clozelIHES}, we denote this intersection by $M$.

Let $\omega$ be a Hecke character of $S(\Q)\backslash S(\AQ)$. We write $\omega_{0}$ for its restriction to $M(\Q)\backslash M(\AQ)$. The space of cuspidal forms $L_{0}^{2}(U(\Q)\backslash U(\AQ),\omega_{0})$ is endowed with an action of 
\[G^{1}:=G(\Q)S(\AQ)U(\AQ)\]
 where $G(\Q)$ acts by conjugation, $S(\AQ)$ acts via $\omega$ and $U(\AQ)$ acts by right translation. We know $G^{1}$ is a closed subgroup of $G(\AQ)$ and the quotient $G(\AQ)/G^{1}$ is compact. The representation of $G(\AQ)$ given by right translation on the cuspidal spectrum $L_{0}^{2}(G(\Q)\backslash G(\AQ),\omega)$ is nothing but \[Ind_{G^{1}}^{G(\AQ)}L_{0}^{2}(U(\Q)\backslash U(\AQ),\omega_{0}).\]
 
 \end{proof}
 
 \begin{rem}
 Let $\pi$ be a cuspidal representation of $GU(\AQ)$. Each constituent in the restriction of $\pi$ to $U(\AQ)$ has the same unramified components. In particular, they all have the same partial $L$-function.
 \end{rem}

\begin{lemma}\label{extend to similitude group}
Let $\pi_{0}$ be an algebraic cuspidal automorphic representation of $U(\AQ)$. We can always extend it to an algebraic cuspidal automorphic representation of $G(\AQ)$. 

Moreover, if $\pi_{0}$ is tempered at some place, discrete series at some place, or cohomological, then its extension has the same property.\end{lemma}

\begin{proof}
To show the existence of the extension, we only need to extend the central character of $\pi_{0}$ to an algebraic Hecke character of $S(\Q)\backslash S(\AQ)$ by the above lemma.

In fact, since $M(\Q)\backslash M(\AQ)$ is compact, the central character of $\pi_{0}$ is always unitary. Hence it lives in the Pontryagin dual of $S(\Q)\backslash S(\AQ)$. We know the Pontryagin dual is an exact functor. Therefore, we can extend it to a unitary Hecke character of $M(\Q)\backslash M(\AQ)$. This unitary Hecke character is not necessarily algebraic. Twisting by a real power of the absolute value, we can get an algebraic Hecke character of $S(\Q)\backslash S(\AQ)$, which is still an extension of the central character of $\pi_{0}$.

To show the extension is locally tempered or discrete series if $\pi_{0}$ is, it is enough to notice that for any place $v$ of $\Q$, $M(\Q_{v})\backslash U(\Q_{v})$ is a finite index subgroup of $S(\Q_{v})\backslash G(\Q_{v})$.

For the cohomological property, we refer to $(5.18)$ of \cite{clozelIHES}.
\end{proof}

\subsection{Base change for unitary groups}
Recall $U$ is the restriction to $\Q$ of the unitary group over $\tr$ associated to $V$. We denote the latter by $U_{0}$. Let $\pi_{0}$ be a cuspidal automorphic representation of $U(\AQ)=U_{0}(\Atr)$. Since $U_{0}(V)(\Acm)\cong GL_{n}(\Acm)$. By Langlands functoriality, we expect to associate $\pi_{0}$ with a $GL_{n}(\Acm)$-representation with expected local components. 

More precisely, we can describe the unramified representations at local non-archimedean places by the Satake parameters. We refer to \cite{minguez} for more details. The local base change can be defined explicitly in terms of the Satake parameters. Let $l/k$ be an extension of local non-archimedean fields and $H$ be a connected reductive group over $k$. The unramified local base change is a map from the set of isomorphism classes of unramified representations of $H(k)$ to that of $H(l)$.

In the global settings, let $L/K$ be an extension of global field and $H$ be a connected reductive group over $K$. We say that an automorphic representation of $H(\mathbb{A}_{L})$ is a weak base change of an automorphic representation of $H(\mathbb{A}_{L})$ if it is the local unramified base change at almost every finite unramified places. If we have multiplicity one theorem for $H(\mathbb{A}_{L})$, for example if $H$ is the unitary group that we will discuss in the following, then the weak base change is unique up to isomorphisms. 

The base change for unitary groups is almost completely clear thanks to Kaletha-Minguez-Shin-White (\cite{kmsw}) and their subsequent articles. But we don't find a precise statement in their paper for our purpose. We use the results and arguments in \cite{labesse}. The following proposition is a slight variation of Th\'{e}or\`{e}me $5.4$ of \cite{labesse}.

\begin{prop}\label{descending}
Let $\Pi$ be a cohomological, conjugate self-dual cuspidal representation of $GL_{n}(\Acm)$. Then $\Pi$ is a weak base change of $\pi_{0}$, a cohomological discrete series representation of $U_{0}(\Atr)$ such that the infinitesimal character of $\Pi_{\infty}$ is compatible with the infinitesimal character of $\pi_{\infty}$ by base change.

We know $\pi_{0}$ is also cuspidal. Moreover, if $\Pi$ has regular highest weight, then so is $\pi_{0}$. In this case, $\pi_{0,\infty}$ is a discrete series representation. 
\end{prop}

\begin{proof}
The existence of $\pi_{0}$ is proved in \cite{labesse}. There are two additional assumptions in the beginning of section $5.2$ of \textit{loc. cit. }, but they are only used for showing multiplicity one in \textit{loc. cit}.

The compatibility of infinitesimal characters is also proved in \textit{loc. cit.} by the calculation on the transfer of Lefschetz function.

We now show that $\pi_{0}$ is cuspidal. Let $v$ be a split place of $F^{+}$ and $w$ be a place of $F$ above $v$ such that $\Pi_{w}$ is the local unramified base change of $\pi_{0,v}$. In particular, we have $U(F^{+}_{v})\cong U(F_{w})\cong GL_{n}(F_{w})$. We know $\Pi_{w}$ is tempered by the Ramanujan conjecture proved in this case by Clozel (\cite{clozelramanujan}) and also by Cariani (\cite{caraianiramanujan}). Hence $\pi_{0,v}$ is tempered since it it isomorphic to $\Pi_{w}$ if we identify $U(F^{+}_{v})$ with $GL_{n}(F_{w})$. The cuspidality then follows from a theorem of Wallach (c.f. \cite{Wallach}) generalized by Clozel (c.f. \cite{clozel93}).

Finally, it is clear that if the highest weight of $\Pi$ is regular then so is $\pi_{0}$. We know a cohomological representation of regular weight is discrete series at infinity by Prop. 4.2 and 5.2 of \cite{lischwermer}.
\end{proof}




  \subsection{Special values of representations of general linear group}\label{GLn}
 Let $\Pi$ be a cohomological conjugate self-dual cuspidal representation of $GL_{n}(\Acm)$.
  
For each $\tau\in \Phi$, let $s_{\tau}$ be an integer in $\{0,1,\cdots,n\}$. We write $I:=(s_{\tau})_{\tau\in\Phi}$ be an element in $\{0,1,\cdots,n\}^{\Phi}$. Let $V_{I}$ be a Hermitian space with respect to $\cm/\tr$ of signature $(n-s_{\tau},s_{\tau})$. We write $U_{0,I}$ for the associated unitary group over $\tr$ and $GU_{I}$ for the associated rational similitude unitary group.

We assume that $\Pi^{\vee}$, the contragredient of $\Pi$, descends by base change to a packet of representations of $U_{0,I}(\Atr)$, which contains a representation $\pi_{0,I}$ satisfying Hypothesis \ref{hypotheses}. 

By Lemma \ref{extend to similitude group}, we can extend $\pi_{0,I}$ to $\pi_{I}$, a cuspidal representation of $GU_{I}(\AQ)$, which still satisfies Hypothesis \ref{hypotheses}.

\begin{rem}
 By Proposition \ref{descending}, we know if $\Pi$ is cohomological with respect to a regular highest weight then it descends by base change to a cuspidal representation of $U_{0,I}(\Atr)$ which is cohomological with respect to a regular highest weight. In particular, this representation satisfies Hypothesis \ref{hypotheses}.
 \end{rem}


\begin{dfn}\label{autoperiods}
Let $\Pi$ be as before. Let $I=(s_{\sigma})_{\sigma\in\Sigma}\in\{0,1,\cdots,n\}^{\Sigma}$. We keep the above notation and define the \textbf{automorphic arithmetic period} $P^{(I)}(\pi)$ by $(2\pi i)^{-2a_{0}}Q_{V_{I}}(\pi)^{-1}$.
\end{dfn}



The following theorem can be deduced directly from Corollary \ref{special value unitary group}.

 \begin{thm}\label{n*1}
 Let $\Pi$ be as before. We denote the infinity type of $\Pi$ at $\tau\in \Phi$ by $(z^{A_{\tau,i}}\overline{z}^{-A_{\tau,i}})_{1\leq i\leq n}$ with $A_{\tau,i}$ in decreasing order for all $\tau\in\Phi$.

 Let $\eta$ be an algebraic Hecke character of $F$. We assume that $\eta^{c}$ can be written in the form $\widetilde{\phi}\alpha$ where $\phi$ is an algebraic Hecke character of $\cm$ of infinity type $z^{-m_{\tau}}\overline{z}^{-m_{\overline{\tau}}}$, $\widetilde{\psi}:=\psi/\psi^{c}$ and $\alpha$ is an algebraic Hecke character of $\cm$ of the infinity type $z^{\kappa}$ and $\tau\in \Phi$.
 
 We suppose that $2m_{\tau}-2m_{\overline{\tau}}-\kappa+2A_{\tau,i}\neq 0$ for all $1\leq i\leq n$ and $\tau\in \Phi$. We define $I:=I(\Pi,\eta)$ to be the map on $\Phi$ which sends $\tau\in\Phi$ to $I(\tau):=\#\{i:2m_{\tau}-2m_{\overline{\tau}}-\kappa+2A_{\tau,i}<0\}$. As before, we write $P^{*,(I(\Pi,\eta))}(\Pi)$ for $(P^{(I(\text{}^{\rho}\Pi,\text{}^{\rho}\eta))}(\text{}^{\rho}\Pi))_{\rho\in J_{E(\Pi,\eta)}} \in E(\Pi,\eta)\otimes \C$.  
 
Let $m\in \Z$ such that  $m\geq n-\cfrac{\kappa}{2}$ and satisfies equation (\ref{mainineq}) with $s_{\tau}=I(\tau)$ and $r_{\tau}=n-I(\tau)$,
then we have:

\begin{eqnarray}
&L^{*}(m-\frac{n}{2},\Pi\otimes \eta) \sim_{E(\Pi,\eta);F^{\Gal}}& \\ \nonumber
 &(2 \pi i)^{(m-n/2)nd(\tr)} \Icm^{[n/2]} (D_{\tr}^{1/2})^{n}e_{\Phi}^{mn}P^{*,(I(\Pi,\eta))}(\Pi) \prod\limits_{\tau\in\Phi}p^{*}(\check{\eta},\tau)^{I(\tau)}p^{*}(\check{\eta},\overline{\tau})^{n-I(\tau)}&
\end{eqnarray}
where $d(\tr)$ is the degree of $\tr$ over $\Q$.

 \end{thm}
 
 \begin{rem}
 \begin{enumerate}
 \item The infinity type stated in the theorem is different from the infinity type in subsection \ref{subsection: algebraic Hecke}. When we say previously that $\psi$, an algebraic Hecke character of $\cm$, is of infinity type $(m_{\tau})_{\tau\in J_{\cm}}$, we mean that $\psi$ is of infinity type $z^{-m_{\tau}}\overline{z}^{-m_{\overline{\tau}}}$ at $\tau\in J_{\cm}$ here. 
 
 \item This theorem is first stated as Theorem $5.2.1$ in \cite{linthesis}. It was proved by assuming a conjecture (c.f. Conjecture $5.1.1$ of \textit{loc. cit.}) which is nothing but a variation of our Theorem \ref{special value unitary group}. 
 
 \item For any algebraic Hecke character $\eta$ of $\cm$, we know $a(\tau)+b(\tau)$ is an integer independent of $\tau$, denoted by $-\omega(\eta)$. It is easy to show that if we allow to change the CM type $\Phi$ then $\eta^{c}$ is of the form $\tilde{\phi}\alpha$ as above if and only if $\omega(\eta)$ is odd, or $\omega(\eta)$ is even and the integers $a(\tau)$ has the same parity for different $\tau$.

 \end{enumerate}
 \end{rem}
 
 \begin{proof}
 Let $I$ be as in the statement of the proof and $\pi_{I}$ be as before.
 We have that (c.f. $3.5.1$ of \cite{harriscrelle})
 \begin{eqnarray}\nonumber
  L^{S,mot}\left(m,\pi_{I}\otimes\psi,\St, \alpha \right) &=& L^{S}\left(m-\frac{n}{2},\pi_{I}\otimes\psi,\St, \alpha \right)\\&=&L^{S}\left(m-\frac{n}{2},BC(\pi_{I}\mid_{U})\otimes\widetilde{\psi}\alpha\right)\nonumber\\
  &=&L^{S}\left(m-\frac{n}{2}, \Pi^{c}\otimes\eta^{c}\right)\nonumber
  \\&=&L^{S}\left(m-\frac{n}{2}, \Pi\otimes\eta\right).
 \end{eqnarray}
 We compare this with Definition \ref{autoperiods} and corollary \ref{special value unitary group}. We conclude by the fact that the term $P^{*}(\phi,\alpha)$ in corollary \ref{special value unitary group} is equivalent to \[\prod\limits_{\tau\in\Phi}p^{*}(\check{\eta},\tau)^{I(\tau)}p^{*}(\check{\eta},\overline{\tau})^{n-I(\tau)}\] by the same calculation as in page $138$ of \cite{harriscrelle}.
 
 To get the Galois equivariant version, we only need to notice that the base change map is commutative with the Galois conjugation as proved in Theorem \ref{commutativity base change}. 
 
\end{proof}

\section{Motivic interpretation}\label{section motivic}
\subsection{The Deligne conjecture}
We firstly recall the statement of the general Deligne conjecture. For details, we refer the reader to Deligne's original paper \cite{deligne79}. We adapt the notation in \cite{harrislin}.

Let $\mathcal{M}$ be a motive over $\Q$ with coefficients in a number field $E$, pure of weight $w$. For simplicity, we assume that if $w$ is even then $(w/2,w/2)$ is not a Hodge type of $\mathcal{M}$. In this case, the motive is critical in the sense of \cite{deligne79}. Deligne has defined two elements $c^{+}(\mathcal{M})$ and $c^{-}(\mathcal{M})\in (E\otimes \C)^{\times}$ as determinants of certain period matrices.

For each $\rho\in J_{E}$, we may define the $L$-function $L(s,\mathcal{M},\rho)$. We write $L(s,\mathcal{M})=L(s,\mathcal{M},\rho)_{\rho\in J_{E}}$. If $L(s,\mathcal{M},\rho)$ is holomorphic at $s=s_{0}$ for all $\rho\in J_{E}$, we may consider $L(s_{0},\mathcal{M})$ as an element in $E\otimes\C$.

\begin{dfn}
We say an integer $m$ is \textbf{critical} for $\mathcal{M}$ if neither $L_{\infty}(\mathcal{M},s)$ nor $L_{\infty}(\check{\mathcal{M}},1-s)$ has a pole at $s=m$ where $\check{\mathcal{M}}$ is the dual of $\mathcal{M}$. We call $m$ a \textbf{critical point} for $\mathcal{M}$.
\end{dfn}

Deligne has formulated a conjecture (c.f. \cite{deligne79}) on special values of motivic $L$-functions as follows.

\begin{conj}(\textbf{the Deligne conjecture})
Let $m$ be a critical point for $\mathcal{M}$. We write $\epsilon$ for the sign of $(-1)^{m}$. We then have: 
\begin{equation}
L(m,\mathcal{M}) \sim_{E} (2\pi i)^{mn^{\epsilon}}c^{\epsilon}(\mathcal{M})
\end{equation}
where $n^{+}$ and $n^{-}$ are two integers depending only on $\mathcal{M}$.
\end{conj}
\begin{rem}
We have assumed that if $w$ is even then $(w/2,w/2)$ is not a Hodge type of $\mathcal{M}$. In this case, $dim_{E}\mathcal{M}$ is even and $n^{+}=n^{-}=dim_{E}\mathcal{M}/2$.
\end{rem}

The following lemma can be deduced easily from $(1.3.1)$ of \cite{deligne79} (for the proof, see Lemma $3.1$ of \cite{linCRcomplet}).

\begin{lemma}\label{critical}
Let $\mathcal{M}$ be a pure motive of weight $w$ as before. We assume that if $w$ is even then $(w/2,w/2)$ is not a Hodge type of $\mathcal{M}$. Let $T(\mathcal{M}):=\{p\mid (p,w-p)\text{ is a Hodge type of }\mathcal{M}\}$. An integer $m$ is critical for $\mathcal{M}$ if and only if:
\[\max\{p\in T(M)\mid p<w/2\}<m\leq \min\{p\in T(M)\mid p>w/2\}  .\]

In other words, an integer $m$ is critical for $\mathcal{M}$ if and only if for any $p\in T(M)$ such that $p>w/2$ we have $m\leq p$, and for any $p\in T(M)$ such that $p<w/2$ we have $m>p$.
In particular, critical values always exist in this case.
\end{lemma}

\bigskip

It is not easy to relate Deligne's periods to geometric objects directly. In \cite{harrisadjoint} and its generalization in \cite{linthesis}, more motivic periods are defined for motives over a CM field. These motivic periods can be related more easily to geometric objects. The Deligne periods are calculated in terms of these new periods in the above two papers and in \cite{harrislin}. 

\subsection{Deligne conjecture for tensor product of motives}

We give a special example of the results in \cite{harrislin} which fits in our main results. 

Let $M$ (resp. $M'$) be a regular motive over $\cm$ with coefficients in a number field $E$ of rank $n$ (resp. rank $1$) and pure of weight $w$.

We first fix $\rho\in J_{E}$ an embedding of the coefficient field. For each $\tau\in J_{F}$, we write the Hodge type of $M$ at $\tau$ (and $\rho$) as $(p_{i}(\tau),q_{i}(\tau))_{1\leq i\leq n}$ with $p_{1}(\tau)>p_{2}(\tau)>\cdot>p_{n}(\tau)$. We know that $q_{i}(\tau)=w-p_{i}(\tau)$. 

We write the Hodge type of $M'$ at at $\tau$ (and $\rho$) as $(p(\tau),q(\tau))$. We assume that for any $i$ and $\tau$, $2p_{i}(\tau)+p(\tau)-q(\tau)\neq 0$.

Let $I(M,M')$ be the map on $\Phi$ which sends $\tau\in\Phi$ to $\#\{i: 2p_{i}(\tau)+p(\tau)-q(\tau)-w>0\}$.

The motivic periods $Q^{(i)}(M,\tau)$, $0\leq i\leq n$ and $Q^{(j)}(M,\tau)$, $0\leq j\leq 1$ are defined in Definition $3.1$ of \cite{harrislin} as elements in $(E\otimes\C)^{\times}$. 

As usual, we identify $E\otimes \C$ with $\C^{J_{E}}$. We write the $\rho$-component of $Q^{(i)}(M,\tau)$ by $Q^{(i)}(M,\tau)_{\rho}$. We define \[Q^{*,I(M,M')}(M):=(\prod\limits_{\tau\in \Phi} Q^{(I(M,M')(\tau))}(M,\tau)_{\rho})_{\rho\in J_{E}}.\] We remark that the index $I(M,M')$ depends implicitly on the embedding $\rho\in J_{E}$. 

Similarly, we write \[Q^{*,(0)}(M,\tau)^{n-I(M',M)(\tau)}:=(Q^{(0)}(M,\tau)^{n-I(M',M)(\tau)}_{\rho})_{\rho\in J_{E}}\in (E\otimes \C)^{\times}.\]

\begin{prop}\label{Deligne for n*1}
The Deligne's periods for the motive $Res_{\cm/\Q}(M\otimes M')$ satisfy:
\begin{eqnarray}
\nonumber &c^{+}Res_{\cm/\Q}(M\otimes M')\sim_{E} (2\pi i)^{-\frac{|\Phi|n(n-1)}{2}}\Icm^{[n/2]} (D_{\tr}^{1/2})^{n}\times&\\
\nonumber& \prod\limits_{\tau\in\Sigma_{\Phi}}Q^{*,I(M,M')}(M) \prod\limits_{\tau\in \Phi} Q^{*,(0)}(M,\tau)^{n-I(M',M)(\tau)}Q^{*,(1)}(M,\tau)^{I(M',M)(\tau)}.&
\end{eqnarray}
Moreover, we have
\[ c^{-}Res_{\cm/\Q}(M\otimes M')\sim_{E} e_{\Phi}^{n}c^{-}Res_{\cm/\Q}(M\otimes M').\]
\end{prop}

\begin{proof}
The proposition follows from Propositions $2.11$ and $3.13$ of \cite{harrislin}. We refer to Definition $3.2$ of \textit{loc. cit.} for the definition of the split index. It is enough to show that 
\[ sp(i,M;M',\tau)=0 \text{ if }i\neq I(M,M')(\tau), sp(I(M,M')(\tau),M;M',\tau)=1, \]
\[sp(0,M';M,\tau)=n-I(M,M')(\tau)\text{ and }sp(1,M';M,\tau)=I(M,M')(\tau).\]

We fix $\tau\in J_{\cm}$. We denote $I(M,M')(\tau)$ by $t$. We have:
\[
p_{1}(\tau)-\frac{p(\tau)+q(\tau)+w}{2}>\cdots>p_{t}(\tau)-\frac{p(\tau)+q(\tau)+w}{2}>-p(\tau)>\]
\[p_{t+1}-\frac{p(\tau)+q(\tau)+w}{2}>\cdots>p_{n}-\frac{p(\tau)+q(\tau)+w}{2}.
\]
Therefore $sp(i,M;M',\tau)=0$ for $i\neq t$ and $sp(t,M;M',\tau)=1$ by the definition of split index. The proof for $sp(0,M';M,\tau)=n-I(M,M')(\tau)$ and $sp(1,M';M,\tau)=I(M,M')(\tau)$ is similar.

We now prove the second part. We use the notation $n_{\tau}(\rho)$ and $e_{\tau}(\rho)=(-1)^{n_{\tau}(\rho)}$ as in Remark $2.2$ of \cite{harrislin}. It is easy to see that $n_{\overline{\tau}}(\rho)=n-n_{\tau}(\rho)$. Let $e=\prod\limits_{\tau\in\Phi} e_{\tau}$ be an element in $(E\otimes \C)^{\times}$. Let $g\in Aut(\C)$. Recall that $e_{g\rho}(\tau)=e_{\rho}(g^{-1}\tau)$ by Remark $A.2$ of \cite{harrislin}. Then \begin{eqnarray}
e(g \rho)&=&\prod\limits_{\tau\in\Phi} e_{\tau}(g\rho)=\prod\limits_{\tau\in\Phi} e_{g^{-1}\tau}(\rho)\nonumber\\
&=&(-1)^{n\#(g^{-1}\Phi)\backslash \Phi}\prod\limits_{\tau\in\Phi}e_{\tau}(\rho)\nonumber\\\nonumber
&=&
(-1)^{n*\#(\Phi\backslash g\Phi)}e(\rho).
\end{eqnarray} Hence $e=\pm e_{\Phi}^{n}$ by the definition of $e_{\Phi}$.

We conclude by the fact that $c^{-}(Res_{\cm/\Q}(M\otimes M'))=e c^{+}(Res_{\cm/\Q}(M\otimes M'))$ by Remark $2.2$ of \cite{harrislin}.

\end{proof}

\subsection{Compatibility of the main results with the Deligne conjecture}
Let $\Pi$ be as in section \ref{GLn}. It is conjectured that the representation $\Pi$ is attached to a motive $M=M(\Pi)$ over $\cm$ with coefficients in $E(\Pi)$ (c.f. Conjecture $4.5$ and paragraph $4.3.3$ of \cite{clozelaa}).

We fix $\rho\in J_{E}$. We write the infinity type of $\Pi$ at $\tau\in \Phi$ as $z^{A_{\tau,i}}\overline{z}^{-A_{\tau,i}}$. Then the Hodge type of $M(\Pi)$ at $\tau$ should be $(-A_{\tau,i}+\frac{n-1}{2},A_{\tau,i}+\frac{n-1}{2})_{1\leq i\leq n}$.

Similarly, we write $M'=M(\eta)$ the conjectural motif associated to $\eta$.

We have:
\begin{equation}\label{shift of center} L(s,M\otimes M')=L(s+\cfrac{1-n}{2},\Pi\times \eta).\end{equation}

We want to compare Theorem \ref{n*1} with the Deligne conjecture. The main difficulty is to compare the automorphic periods with the motivic periods. Recall that the automorphic periods $P^{(I)}(\Pi)$ are constructed from different geometric objects. It is hard to relate them with the same motive $M(\Pi)$. However, if we admit the Tate conjecture, we will get \begin{equation}\label{compare periods}
P^{(I)}(\Pi)\sim_{E(\Pi)} Q^{(I)}(\Pi)
\end{equation} as in section $4.4$ of \cite{harrislin}. We recall that roughly speaking, the Tate conjecture says that a motive is determined by its $l$-adic realization.

\begin{coro}
We keep the notation and conditions as in Theorem \ref{n*1}. If we admit the Tate conjecture, then the Deligne conjecture is true up to $\sim_{E(\Pi,\eta);F^{\Gal}}$ for critical values $m>n-\kappa /2$ of the conjectural motive $Res_{\cm/\Q}M(\Pi)\otimes M(\eta)$.
\end{coro}

\begin{proof}
We compare Proposition \ref{Deligne for n*1}, Theorem \ref{n*1}  equation (\ref{shift of center}), equation (\ref{compare periods}) and the fact that:
\[Q^{(0)}(M(\eta),\tau)\sim_{E(\tau)} p(\check{\eta^{c}},\tau), \text{   and    } Q^{(1)}(M(\eta),\tau) \sim_{E(\tau)} p(\check{\eta},\tau)\] by Lemma $3.17$ of \cite{harrislin}. It is easy to verify that $I(\Pi,\eta)=I(M(\Pi),M(\eta)).$ It remains to show that if $m>n-\kappa/2$ critical for $M(\Pi)\otimes M(\eta)$ then $m$ satisfies equation (\ref{mainineq}).

It is trivial that $m\geq \cfrac{n-\kappa}{2}$. We need to show that for all $\tau\in\Phi$, $m \leq -a_{\tau,s_{\tau}+1}+s_{\tau}+m_{\tau}-m_{\bar\tau}-\kappa$ and $m \leq a_{\tau,s_{\tau}}+r_{\tau}+m_{\bar\tau}-m_{\tau}$.

The Hecke character $\eta$ is of infinity type $z^{m_{\tau}-m_{\bar{\tau}}}\overline{z}^{-m_{\tau}+m_{\bar{\tau}}+\kappa}$. Hence the motive $Res_{\cm/\Q}M(\Pi)\otimes M(\eta)$ is of weight $-\kappa+n-1$. Moreover, the set $T(Res_{\cm/\Q}M(\Pi)\otimes M(\eta))$ defined in Lemma \ref{critical} equals to \[
\{-A_{\tau,i}+\cfrac{n-1}{2}+m_{\bar{\tau}}-m_{\tau},  A_{\tau,i}+\cfrac{n-1}{2}+m_{\tau}-m_{\bar{\tau}}-\kappa\mid 1\leq i\leq n, \tau\in \Phi\}.
\]

By the statement of Theorem \ref{n*1}, we have $s_{\tau}=I(\tau):=\#\{i:2m_{\tau}-2m_{\overline{\tau}}-\kappa+2A_{\tau,i}<0\}$. Recall that $A_{\tau,i}$ is in decreasing order. Hence we know that $2m_{\tau}-2m_{\overline{\tau}}-\kappa+2A_{\tau,r_{\tau}+1}<0$ and $2m_{\tau}-2m_{\overline{\tau}}-\kappa+2A_{\tau,r_{\tau}}>0$.

Consequently, we have that $A_{\tau,r_{\tau}}+\cfrac{n-1}{2}+m_{\tau}-m_{\bar{\tau}}-\kappa>\cfrac{n-1-\kappa}{2}$ where the left hand side is an element inside $T(Res_{\cm/\Q}M(\Pi)\otimes M(\eta))$. By Lemma \ref{critical}, the critical value $m$ satisfies $m\leq A_{\tau,r_{\tau}}+\cfrac{n-1}{2}+m_{\tau}-m_{\bar{\tau}}-\kappa$. We recall that $A_{\tau,n+1-i}=-a_{\tau,i}-\cfrac{n+1}{2}+i$ for all $i$. Therefore
\begin{eqnarray}\nonumber
m &\leq &-a_{\tau,s_{\tau}+1}-\cfrac{n+1}{2}+(s_{\tau}+1)+\cfrac{n-1}{2}+m_{\tau}-m_{\bar{\tau}}-\kappa\\ \nonumber
&=&-a_{\tau,s_{\tau}+1}+s_{\tau}+m_{\tau}-m_{\bar{\tau}}-\kappa.\end{eqnarray} 

Similarly, we consider the Hodge number $-A_{\tau,r_{\tau}+1}+\cfrac{n-1}{2}+m_{\bar{\tau}}-m_{\tau}$ which is bigger than $\cfrac{n-1-\kappa}{2} $. We deduce that $m\leq -A_{\tau,r_{\tau}+1}+\cfrac{n-1}{2}+m_{\bar{\tau}}-m_{\tau}$. The right hand side equals to \[a_{\tau,s_{\tau}}+\cfrac{n-1}{2}-s_{\tau}+\cfrac{n-1}{2}+m_{\bar{\tau}}-m_{\tau}=a_{\tau,s_{\tau}}+r_{\tau}+m_{\bar{\tau}}-m_{\tau}\] as expected.

\end{proof}

\appendix

\section{Commutativity of Galois conjugation and base change}
In this appendix, we prove the commutativity of Galois conjugation and base change for unitary groups.
\subsection{Notation}
Let $l/k$ be an unramified extension of local non-archimedean fields. Let $G$ be a reductive group over $k$. Let $P$ be a minimal parabolic subgroup of $G(k)$ and $M$ be a corresponding Levi factor.

Let $\chi$ be an unramified character of $M$. We may regard it as a representation of $P$. The unitarily parabolic induction is defined as
$i_{P}^{G}(\chi) :=$
\[ \{\phi: G(k)\rightarrow \C \text{ continuous} : \phi(pg)=\delta_{P}^{1/2}(p)\chi(p)\phi(g)\text{, } p\in P\text{, }g\in G(k)\}
\]
where $\delta_{P}$ is the modulus character of $P$ (see \cite{minguez}).

The unitarily parabolic induction gives rise to a surjective map from the set of unramified characters of $M$ to the set of isomorphism classes of unramified representations of $G(k)$. Two unramified characters induce the same $G(k)$-representation if and only if they are equivalent under the action of the Weyl group.

\subsection{Conjugation of representations over local non-archimedean fields}
Let $\sigma\in\Aut(\C)$. Let $V$ be a complex representation of $G(k)$. We let ${}^{\sigma}V=V\otimes_{\C,\sigma}\C$, with $G(k)$ acting on the first factor.

Let $\chi$ be as before. If $\phi\in i_{P}^{G}(\chi)$, we have $(\sigma\circ \phi)(pg)=\sigma(\delta_{P}^{1/2}(p)){}^{\sigma} \chi(g)(\sigma\circ \phi )(g)$ for any $p\in P$ and $g\in G(k)$.

We define the character $T_{\sigma}$ on $P$ by
\[ T_{\sigma}(p)=\cfrac{\sigma(\delta_{P}^{1/2}(p))}{\delta_{P}^{1/2}(p)}.\]
It is easy to see that ${}^{\sigma}(i_{P}^{G}(\chi))\simeq i_{P}^{G}(T_{\sigma}*\chi)$.

\subsection{Local base change}

The unramified base change map sends isomorphism classes of unramified representations of $G(k)$ to those of $G(l)$. It can be defined naturally through the dual group. We refer to \cite{minguez} for more details. 

We write $P_{l}$ for a minimal parabolic subgroup of $G(l)$, and $M_{l}$ for a corresponding Levi factor. The base change map induces a map from the set of equivalence classes of characters of $M$ (under the action of the Weyl group) to those of $M_{l}$. We write $[\chi]$ for the equivalence class of $\chi$. We take $\chi_{l}$ a character in the equivalence class of the image of $[\chi]$.

We can define the character $T_{l,\sigma}$ of $P_{l}$ in a similar fashion as $T_{\sigma}$. It is natural to expect the commutativity of local base change and Galois conjugation, or equivalently:
\begin{equation}\label{commutativity}
[T_{\sigma,l}*\chi_{l}]=[(T_{\sigma}*\chi)_{l}]
\end{equation}

We will show that this is true if $l$ is a quadratic base change and $G$ is the quasi-split unitary group of rank $n$ with respect to $l/k$ in the next subsection.

\subsection{Commutativity for quasi-split unitary groups}
Let $n$ be an integer. We assume that $n=2m$ is even for simplicity. The case $n=2m+1$ is similar and we leave it to the reader.

We take $U$ to be the quasi-split unitary group of rank $n$ with respect to $l/k$ defined over $k$. Choosing a proper basis, we may identify $U(k)$ with
\[ \left\{ X\in GL_{n}(l): \text{}^{t}\overline{X}
\begin{pmatrix}
0 & I_{m}\\
-I_{m} & 0 
\end{pmatrix}
X=\begin{pmatrix}
0 & I_{m}\\
-I_{m} & 0
\end{pmatrix}
\right\}. \]
We let $P$ be the minimal parabolic given as the intersection of $U(k)$ with the set of upper triangular matrices in $GL_{n}(l)$.

Let $P_{0}$ be the algebraic group defined over $k$ consisting of upper triangular matrices in $GL_{m}(l)$. Let $S$ be the algebraic group defined over $k$ such that $S(k)=\{X\in M_{m}(l): {}^{t}\overline{X}=X \}$.

The parabolic group $P$ consists of elements of the form
\[\begin{pmatrix}
g & gX\\
0 & {}^{t}\overline{g}^{-1}
\end{pmatrix},\]
where $g\in P_{0}(k)$ and $X\in S(k)$.

Let $\diff_{L}g$ (resp. $\diff_{R}g$) be a left (resp. right) invariant Haar measure on $P_{0}(k)$ and $\diff X$ be a (left and right) invariant Haar measure on $S(k)$.  We may assume that $\diff_{R}g=\delta_{P_{0}}^{-1}\diff_{L}(g)$.

Let $\begin{pmatrix}
A & AB\\
0 & {}^{t}\overline{A}^{-1}
\end{pmatrix} \in P$. We have that
\[\begin{pmatrix}
A & AB\\
0 & {}^{t}\overline{A}^{-1}
\end{pmatrix}
\begin{pmatrix}
g & gX\\
0 & {}^{t}\overline{g}^{-1}
\end{pmatrix}
=
\begin{pmatrix}
Ag & Ag(X+g^{-1}B{}^{t}\overline{g}^{-1})\\
0 & {}^{t}\overline{Ag}^{-1}
\end{pmatrix}
\]

and

\[
\begin{pmatrix}
g & gX\\
0 & {}^{t}\overline{g}^{-1}
\end{pmatrix}
\begin{pmatrix}
A & AB\\
0 & {}^{t}\overline{A}^{-1}
\end{pmatrix}
=
\begin{pmatrix}
gA & gA(B+A^{-1}X{}^{t}\overline{A}^{-1})\\
0 & {}^{t}\overline{gA}^{-1}
\end{pmatrix}
\]
It is easy to verify that $\diff_{L}g\diff X$ is a left invariant Haar measure and 
\[ |\det(g)|_{k}^{2m}\diff_{R}g\diff X\]
is a right invariant Haar measure on $P$. We obtain that \[
\delta_{P}\begin{pmatrix}
g & X\\
0 & {}^{t}\overline{g}^{-1}
\end{pmatrix} =\delta_{P_{0}}(g)|\det(g)\det(\bar{g})|_{k}^{m}=\delta_{P_{0}}(g)|\det(g)|_{l}^{m}.\]
The last equation is due to the fact that $l/k$ is unramifield. In the following, we write $|\cdot|$ for the absolute value in $l$.

We write the diagonal of $g$ as $(g_{1},\cdots,g_{m})$. Then
\[
\delta_{P_{0}}^{1/2}(g)=|g_{1}|^{\frac{m-1}{2}}|g_{2}|^{\frac{m-3}{2}}\cdots|g_{m}|^{-\frac{m-1}{2}}.
\]
Therefore, $\delta_{P}^{1/2}(g)=|g_{1}|^{\frac{2m-1}{2}}|g_{2}|^{\frac{2m-3}{2}}\cdots|g_{m}|^{\frac{1}{2}}$.

We now consider $U(l)\cong GL_{n}(l)$. We take $P_{l}$ to be the minimal parabolic subgroup of $U(l)$ consisting of upper triangular matrices. Let $p_{l}\in U(l)$ with diagonal $(g_{1},\cdots,g_{2m})$.  By Theorem 4.1 of \cite{minguez}, we have that
\[\chi_{l}(p_{l})=\chi((g_{1},\cdots,g_{m})) \chi((g_{m+1},\cdots,g_{2m}))^{-1}
 \]
for any character $\chi$. (Here we consider the first case in Theorem 4.1 of \textit{op. cit.}. The proof for the second case is similar.)

We can see easily that $(T_{\sigma}*\chi)_{l}=(T_{\sigma})_{l}*\chi_{l}$. Thus, to show (\ref{commutativity}), it is enough to show that $(T_{\sigma})_{l}= T_{\sigma,l}$. In fact, both sides map $p_{l}$ to \[
\left(\cfrac{\sigma(|g_{1}|)}{|g_{1}|}\right)^{\frac{2m-1}{2}}\left(\cfrac{\sigma(|g_{2}|)}{|g_{2}|}\right)^{\frac{2m-3}{2}}\cdots \left(\cfrac{\sigma(|g_{2m}|)}{|g_{2m}|}\right)^{-\frac{2m-1}{2}}.
\]

We have deduced that:
\begin{prop}\label{local commutativity}
Let $\pi$ be an unramified representation of the quasi-split unitary group $U(k)$. We write $BC(\pi)$ for its unramified base change to $U(l)$. Then for any $\sigma\in Aut(\C)$, the base change of ${}^{\sigma} \pi$ to $U(l)$ is isomorphic to ${}^{\sigma} (BC(\pi))$, i.e. $BC({}^{\sigma} \pi)\cong {}^{\sigma} (BC(\pi))$.
\end{prop}

\subsection{Global base change}
The commutativity of base change and Galois conjugation for local unramified representations implies the commutativity for automorphic representations.

\begin{thm}\label{commutativity base change}
Let $\cm/\tr$ be a quadratic extension of number fields . Let $U$ be a unitary group of rank $n$ with respect to the extension $\cm/\tr$. Let $\pi$ be an automorphic representation of $U(\Atr)$. We assume that for any $\sigma\in Auc(\C)$ there exists an automorphic representation ${}^{\sigma}\pi$ of ${}^{\sigma}U(\Atr)$ such that $({}^{\sigma}\pi)_{f}\cong{}^{\sigma}\pi_{f}$.

We assume that the (weak) base change of $\pi$ to $GL_{n}(\Acm)$ exists and denote it by $BC(\pi)$. We also assume that ${}^{\sigma}(BC(\pi))$, the Galois conjugation of $BC(\pi)$ by $\sigma$, exists for any $\sigma\in Auc(\C)$.

Then ${}^{\sigma}(BC(\pi))$ is the (weak) base change of of ${}^{\sigma}\pi$, i. e. $BC({}^{\sigma} \pi)\cong {}^{\sigma} (BC(\pi))$.
\end{thm}
\begin{proof}
For almost all finite places $v$ of $\tr$, we have that $U_{v}$ is quasi-split, $\pi_{v}$ is unramified and $BC(\pi)_{v}\cong BC(\pi_{v})$. By Proposition \ref{local commutativity}, we know ${}^{\sigma}(BC(\pi_{v}))\cong BC({}^{\sigma}\pi_{v})$. Hence $({}^{\sigma}BC(\pi))_{v} \cong {}^{\sigma}(BC(\pi)_{v})\cong BC({}^{\sigma}\pi_{v})$ for almost all finite places $v$. In other words, ${}^{\sigma}BC(\pi)$ is the (weak) base change of of ${}^{\sigma}\pi$.
\end{proof}

\begin{rem}
\begin{enumerate}
\item We have discussed when the Galois conjugation of a unitary group representation exists. We remark that for a regular algebraic isobaric representation of $GL_{n}$, its Galois conjugation exists by Théorème $3.13$ of \cite{clozelaa}.

\item The commutativity of Galois conjugation and the Jacquet-Langlands transfer is proved in \cite{grobner-raghuram}.
\end{enumerate}
\end{rem}

\bibliography{galoisequiv.bib} \bibliographystyle{amsalpha}

\end{document}